\documentclass{amsart}
\usepackage{amsfonts,amssymb,amscd,amsmath,enumerate,verbatim,calc}
\usepackage[all]{xy}
\usepackage{euscript}

\newcommand{\CM}{Cohen-Macaulay}

\newcommand{\wrt}{with respect to}

\newcommand{\D}{\mathcal{D} }

\newcommand{\bx}{\mathbf{x} }
\newcommand{\n}{\mathfrak{n} }
\newcommand{\m}{\mathfrak{m} }

\newcommand{\q}{\mathfrak{q} }

\newcommand{\Z}{\mathbb{Z} }
\newcommand{\Q}{\mathbb{Q} }

\newcommand{\eR}{\mathcal{R} }
\newcommand{\Fc}{\mathcal{F} }
\newcommand{\Hc}{\mathcal{H} }

\newcommand{\C}{\mathcal{C}}
\newcommand{\Dc}{\mathcal{D}}
\newcommand{\Kc}{\mathcal{K} }
\newcommand{\Ic}{\mathcal{I} }

\newcommand{\Tc}{\mathcal{T} }
\newcommand{\Fb}{\mathbb{F} }
\newcommand{\Gc}{\mathcal{G} }

\newcommand{\rt}{\rightarrow}

\newcommand{\ov}{\overline}

\newcommand{\Om}{\Omega}

\newcommand{\wh}{\widehat }
\newcommand{\wt}{\widetilde }

\newcommand{\image}{\operatorname{image}}
\newcommand{\cone}{\operatorname{cone}}

\newcommand{\depth}{\operatorname{depth}}
\newcommand{\charp}{\operatorname{char}}

\newcommand{\coker}{\operatorname{coker}}
\newcommand{\cx}{\operatorname{cx}}
\newcommand{\rank}{\operatorname{rank}}
\newcommand{\curv}{\operatorname{curv}}
\newcommand{\rad}{\operatorname{rad}}
\newcommand{\height}{\operatorname{height}}
\newcommand{\htt}{\operatorname{height}}

\newcommand{\Spec}{\operatorname{Spec}}
\newcommand{\Supp}{\operatorname{Supp}}

\newcommand{\CMS}{\operatorname{\underline{CM}}}

\newcommand{\CMa}{\operatorname{CM}}

\newcommand{\projdim}{\operatorname{projdim}}
\newcommand{\Hom}{\operatorname{Hom}}
\newcommand{\md}{\operatorname{mod}}
\newcommand{\smod}{\operatorname{^*mod}}

\newcommand{\sHom}{\operatorname{\underline{Hom}}}
\newcommand{\apc}{\operatorname{\underline{APC}}}
\newcommand{\gapc}{\operatorname{\underline{^*APC}}}
\newcommand{\Ext}{\operatorname{Ext}}
\newcommand{\Tor}{\operatorname{Tor}}
\newcommand{\thick}{\operatorname{thick}}
\newcommand{\Min}{\operatorname{Min}}
\theoremstyle{plain}

\newtheorem{theorem}{Theorem}[section]
\newtheorem{corollary}[theorem]{Corollary}
\newtheorem{lemma}[theorem]{Lemma}
\newtheorem{proposition}[theorem]{Proposition}

\theoremstyle{definition}
\newtheorem{definition}[theorem]{Definition}
\newtheorem{s}[theorem]{}
\newtheorem{remark}[theorem]{Remark}

\newtheorem{construction}[theorem]{Construction}
\theoremstyle{remark}

\begin{document}

\title[Stable category]{On the Stable category of maximal Cohen-Macaulay modules over Gorenstein rings-II }
\author{Tony~J.~Puthenpurakal}
\date{\today}
\address{Department of Mathematics, IIT Bombay, Powai, Mumbai 400 076}

\email{tputhen@math.iitb.ac.in}

 \begin{abstract}
Let $(A,\m), (B,\n) $ be  Gorenstein local rings  and let $\CMS(A)$ be its stable category of finitely generated maximal \CM \ $A$-modules. Suppose we have an equivalence $\Phi \colon \CMS(A) \rt \CMS(B)$ as triangulated categories. We show
\begin{enumerate}
  \item If $A, B$ are not hypersurfaces then $\dim A = \dim B$.
  \item If $M$ is a maximal \CM \ $A$-module then $$\curv_A(M) = \curv_B (\Phi(M)),$$ here $\curv_A(M) = \limsup_n \sqrt[n]{\ell(\Tor^A_n(M,k))}$.
  \item $A$ satisfies Serre's condition $R_i$ if and only if $B$ satisfies $R_i$.
  \item $A$ is a complete intersection on the punctured spectrum of $A$ if and only if $B$ is a complete intersection on the punctured spectrum of $B$.
\end{enumerate}
We also show that $\Phi$ imposes constraints of residue field of $B$ in terms of residue field of $A$ and vice-versa. Finally if $I$ is an ideal in $A$ such that the extended Rees algebra $\eR(I) = A[It, t^{-1}]$ is Gorenstein then we construct a triangulated functor $\Psi \colon \CMS^\Z(\eR(I)) \rt \CMS(A)$ where $\CMS^\Z(\eR(I))$ is the stable category of all graded maximal \CM \ $\eR(I)$-modules. We show that $\Psi$ induces an equivalence $\CMS^\Z(\eR(I))/\ker \Psi \rt \CMS(A)$. We give some applications of this map.
\end{abstract}
 \maketitle
\section{introduction}
Dear Reader, while reading this paper it is a good idea to have part 1 of this paper \cite{P-stable}.
Following Auslander's insight we study properties of triangle equivalences between the stable category of MCM modules over Gorenstein ring.

Let $(A,\m)$ be a  commutative Gorenstein local ring with residue field $k$. Let $\CMa(A)$ denote the full subcategory of finitely generated  MCM (= maximal \CM) $A$-modules  and let $\CMS(A)$ denote the stable category of MCM $A$-modules. Recall that objects in $\CMS(A)$ are same as objects in $\CMa(A)$. However the set of morphisms $\underline{\Hom_A}(M,N)$ between $M$ and $N$ is  $= \Hom_A(M,N)/P(M,N)$ where $P(M,N)$ is the set of $A$-linear maps from $M$ to $N$ which factor through a finitely generated free module. It is well-known that $\CMS(A)$ is  a triangulated category with translation functor $\Omega^{-1}$,  (see \cite{Bu}).  Here $\Omega(M)$ denotes the syzygy of $M$ and $\Omega^{-1}(M)$ denotes the co-syzygy of $M$. Also recall that an object $M$ is zero in $\CMS(A)$ if and only if it is free  considered as an $A$-module. Furthermore $M \cong N$ in $\CMS(A)$ if and only
if there exists finitely generated free modules $F,G$ with $M\oplus F \cong N \oplus G$ as $A$-modules.  If $A$ is regular local then all MCM
modules are free. So $\CMS(A) = 0$.

 We use Neeman's book \cite{N} for notation on triangulated categories. However we will assume that if $\mathcal{C}$ is a triangulated category then $\Hom_\mathcal{C}(X, Y)$ is a set for any objects $X, Y$ of $\mathcal{C}$.

\textbf{I:} Our first result states that in most cases dimension of a Gorenstein local ring is a stable invariant. More precisely:
\begin{theorem}
\label{first}
Let $(A,\m), (B,\n) $ be  Gorenstein local rings. Suppose we have an equivalence $\Phi \colon \CMS(A) \cong \CMS(B)$ as triangulated categories. We have
\begin{enumerate}[\rm (1)]
  \item If $A, B$ are not hypersurfaces then $\dim A = \dim B$.
  \item  If $A$ (and so $B$)  is a hypersurface  singularity and if the completion of $A$ has a MCM module $M$ with $\Omega(M) \ncong M$  and $M$ free on the punctured spectrum of $A$ then $\dim A - \dim B$ is even.
 \end{enumerate}
\end{theorem}
{In part-1 of the paper we had proved the result when $A$ and $B$ are Henselian, see \cite[1.4]{P-stable}.

\textbf{II:} To state our next result we need a few preliminaries. If $M$ is a finitely generated  $A$-module then set $\beta_i(M) = \ell(\Tor^A_i(M,k))$ for $i \geq 0$. The \emph{complexity} of $M$ is defined as
$$\cx_A(M) = \inf\{ r \in \mathbb{N} \mid  \limsup_{n \rt \infty} \frac{\beta_n(M)}{n^{r - 1}}  < \infty \}.$$
It is possible that $\cx_A(M)$ is infinite. To deal with this the following notion was introduced in \cite{A-ext}:
$$\curv_A(M) = \limsup_{n \rt \infty} \sqrt[n]{\beta_n(M)}.$$
It can be shown that $\curv(M) \leq \curv(k) < \infty$ where $k$ is the residue field of $A$; see \cite[4.2.3, 4.2.4]{A}. Our second result is
\begin{theorem}
\label{second}Let $(A,\m), (B,\n) $ be  Gorenstein local rings. Suppose we have an equivalence $\Phi \colon \CMS(A) \cong \CMS(B)$ as triangulated categories. Then
\begin{enumerate}[\rm (1)]
  \item $\curv_A(M) = \curv_B(\Phi(M))$ for any $M \in \CMS(A)$.
  \item $\curv_A(A/\m) = \curv_B(B/\n)$.
  \item If $M \in \CMS(A)$ then $\cx_A(M)$ is infinite if and only if $\cx_B(\Phi(M))$ is infinite.
  \item If $M \in \CMS(A)$ and $\cx_A(M)$ is finite then $\cx_A(M) = \cx_B(\Phi(M))$.
\end{enumerate}
\end{theorem}

\textbf{III:} The residue field $k$ of $A$ has dimension zero. So it is far from being MCM. Nevertheless surprisingly a stable equivalence $\Phi \colon \CMS(A) \rt \CMS(B)$ imposes constraints on the residue field of $A$ in terms of residue field of $B$ and vice-versa.
\begin{theorem}
\label{third}Let $(A,\m), (B,\n) $ be  Gorenstein local rings. Suppose we have an equivalence $\Phi \colon \CMS(A) \cong \CMS(B)$ as triangulated categories. Assume $A$ (and so $B$) is not regular. Let $k$ be the residue field of $A$ and let $l$ be the residue field of $B$. Then
\begin{enumerate}[\rm (1)]
  \item $\charp k = \charp l$.
  \item $k$ is finite if and only if $l$ is finite.
  \item If $k$ is infinite then both $k$ and $l$ have the same cardinality.
  \item $k$ is perfect if and only if $l$ is perfect.
  \item Assume $k$ is algebraically closed.
  Then
  \begin{enumerate}[\rm (a)]
    \item If $\charp k > 0$ then $l \cong k$.
    \item If $\charp k = 0$ then either $l \cong k$ or $l$ is a real closed field and $l(\sqrt{-1}) \cong k$.
  \end{enumerate}
  \end{enumerate}
\end{theorem}

\textbf{IV:} There are not many examples of triangulated functors from stable categories. One is given by flat maps, the second is given by specialization with respect to regular sequences,  Kn\"orrer periodicity, \cite{K} and the map giving MCM approximations, \cite[p.\ 740]{BJM}. To this meager list we add another triangulated functor. Let $I \subseteq \m$ be an ideal in $A$ and let  $\eR_A(I) = A[It, t^{-1}] = \bigoplus_{n \in \Z}I^n$ be the extended Rees algebra of $I$ and let $G_I(A) = \bigoplus_{n\geq 0}I^n/I^{n+1}$ be the associated graded ring of $I$.  Note $t^{-1}$ is $\eR_A(I)$-regular and $\eR_A(I)/(t^{-1})  = G_I(A)$. There are bountiful examples of $\m$-primary ideals with $\eR_A(I)$-Gorenstein, see \cite[1.5]{P-Itoh-n}.
If $\eR_A(I)$ is Gorenstein then let $\CMS_\Z(\eR_A(I))$ denotes the category of $\Z$-graded MCM $\eR_A(I)$-modules.
We prove:
\begin{theorem}
\label{fourth} Let $(A,\m)$ be a Gorenstein local ring and let $I \subseteq \m$ be an ideal such that $\eR_A(I)$ is Gorenstein. Then  there exists a triangulated functor \\ $\Psi \colon \CMS_\Z(\eR_A(I)) \rt \CMS(A)$ such that $\Psi$ induces an equivalence
$$\ov{\Psi} \colon \frac{\CMS_\Z(\eR_A(I))}{\ker \Psi} \cong \CMS(A).$$
\end{theorem}

\textbf{V:} Let $T = \bigoplus_{n \in \Z}T_n$ be a graded Noetherian ring. Let $\smod(T)$ denote the abelian category of finitely generated graded $T$-modules. Let $G_0(\smod(T))$ denotes its zeroth $K$-group.
Then there is an action of $\Z[u,u^{-1}]$ on $G_0(\smod(T))$  defined as $u [M] = [M[1]]$. If $T$ is Gorenstein then similarly there is a $\Z[u,u^{-1}]$ module structure on $G_0(\CMS_\Z(T))$ the  zeroth Grothendieck group of the stable category of all graded MCM $T$-modules.
Let $\Q(u)$ be the Quotient field of $\Z[u,u^{-1}]$. If $E$ is a $\Z[u, u^{-1}]$-module then let $E_{\Q(u)} = E \otimes_{\Z[u, u^{-1}]}\Q(u)$.
The analysis of the proof of Theorem \ref{fourth}  yields the following result.

\begin{theorem}
\label{fifth} Let $(A,\m)$ be a Gorenstein local ring and let $I \subseteq \m$ be an ideal such that $\eR_A(I)$ is Gorenstein. Then the MCM approximation functor \\ $X \colon \CMS_\Z(G_I(A)) \rt
\CMS_\Z(\eR_A(I))$ yields an isomorphism
$$G_0(\CMS_\Z(G_I(A))_{\Q(u)} \rt G_0(\CMS_\Z(\eR_A(I)))_{\Q(u)}.$$
\end{theorem}
We use Theorem \ref{fifth} to show that in many cases $G_0(\CMS_\Z(\eR_A(I)))_{\Q(u)}$ is a finite dimensional $\Q(u)$-vector space.

\textbf{VI:}
Let $(A,\m)$ be a Gorenstein local ring and let $I$ be an $\m$-primary ideal with $\eR(I)$ Gorenstein. Let $\wh{A}$ be the completion of $A$ and let
 $\wh{\eR(I)}= \bigoplus_{n \in \Z}I^n\wh{A} = \eR(I)\otimes_A \wh{A}$ be the extended Rees algebra of the ideal $I\wh{A}$.
 We have a natural functor $\Phi \colon \CMS_\Z(\eR(I)) \rt \CMS_\Z(\wh{\eR(I)})$ defined by sending $E \rt E\otimes_A \wh{A} = E\otimes_{\eR(I)}\wh{\eR(I)}$.
 Set $\CMS^0_\Z(\eR(I)) $ to be the thick subcategory of $\CMS_\Z(\eR(I))$ consisting of graded MCM $\eR(I)$-modules which are free on punctured homogeneous spectrum of $\eR(I)$.
\begin{theorem}
\label{equi}(with hypotheses as in \textbf{VI}).  $\Phi$ induces an equivalence  $$\Phi^\sharp \colon \CMS^0_\Z(\eR(I)) \rt \CMS^0_\Z(\wh{\eR(I)}).$$
\end{theorem}
We note that it is particularly difficult to construct equivalences of triangulated categories in commutative algebra. Theorem \ref{equi} constructs a large class of equivalences.

\textbf{VII:}
Recall a Noetherian ring $R$ is said to have $R_i$ property of Serre if $R_P$ is regular for all primes $P$ with $\htt P \leq i$. We prove:
\begin{theorem}
\label{ri}
Let $(A,\m), (B,\n) $ be  Gorenstein local rings. Suppose we have an equivalence $\Phi \colon \CMS(A) \cong \CMS(B)$ as triangulated categories. Assume $A$ (and so $B$) is not a hypersurface ring. Then $A$ satisfies $R_i$ if and only if $B$ satisfies $R_i$.
\end{theorem}

\begin{remark}
Theorem \ref{ri} does not hold when $A$ is a hypersurface ring. For instance Kn\"orrer periodicity gives an equivalence of stable categories of \\ $A = \mathbb{C}[[X, Y]]/(Y^2)$ and
$B = \mathbb{C}[[X, Y, Z_1, Z_2]]/(Y^2 + Z_1^2 + Z_2^2)$. We note that $B$ is $R_1$ but $A$ is \emph{not} reduced.
\end{remark}
In part one of the paper we showed that if $\Phi \colon \CMS(A) \cong \CMS(B)$ is an equivalence then $A$ is an abstract complete intersection of codimension $c$ if and only if $B$ is an abstract complete intersection of codimension $c$. We say $A$ is a complete intersection on the punctured spectrum if $A_P$ is an abstract complete intersection for all prime $P \neq \m$.

We prove:
\begin{theorem}
\label{ci}
Let $(A,\m), (B,\n) $ be  Gorenstein local rings. Suppose we have an equivalence $\Phi \colon \CMS(A) \cong \CMS(B)$ as triangulated categories. Then $A$  is a complete intersection on the punctured spectrum of $A$ if and only if $B$ is a complete intersection on the punctured spectrum of $B$.
\end{theorem}

\textbf{VIII:} Let $\C$ be a Krull-Remak -Schmidt (KRS) traingulated category and let $\D$ be a \emph{proper} thick subcategory of $\C$. Let $M \in \C$. Let $M = M_1^{a_1}\oplus \cdots \oplus M_n^{a_n}$ with $M_i$ indecomposable and $M_ i \ncong M_j$ for $i\neq j$; also $a_i \geq 1$. We define $\alpha_{\C, \D}(M)$ as follows:

(1) if each $M_i \in \D$ then set $\alpha_{\C, \D}(M) = 0$.

(2) if $M_{i_1}, \ldots, M_{i_r} \notin \D$ and $M_j \in \D$ for $j \neq i_l$ then set $\alpha_{\C, \D}(M) = a_{i_1} + \cdots + a_{i_r}.$

For $M \in \C$ define

$r(\C, \D, M) =  \{  t_s \mid t_s \colon N_s \rt E_s \rt M \rt N_s[1] \ \text{with} \ t_s  \ \text{a triangle and } \ N_s \in \D  \},$

and

$l(\C, \D, M) = \{ t_s \mid t_s \colon M \rt E_s \rt N_s \rt M[1] \ \text{with} \ t_s  \ \text{a triangle and }  N_s \in \D  \}.$

If $t_s \in r(\C, \D, M)$ with $t_s \colon N_s \rt E_s \rt M \rt N_s[1]$ then define \\ $\alpha_{\C, \D}(t_s) = \alpha_{\C, \D}(E_s)$. Analogously define $\alpha_{\C, \D}(t_s)$ for $t_s \in l(\C, \D, M)$.

Set $r_0(\C, \D, M) = \sup \{ \alpha_{\C, \D}(t_s) \mid t_s \in r(\C, \D, M) \}$. Also set

$l_0(\C, \D, M) = \sup \{ \alpha_{\C, \D}(t_s) \mid t_s \in l(\C, \D, M) \}$.

Note we are NOT claiming that $r_0(\C,\D, M)$ or $l_0(\C, \D, M)$ are finite.  We say that $(\C, \D)$ is \emph{rigid} pair if $r_0(\C,\D, M)$ and $l_0(\C, \D, M) $ are finite for all $M \in \C$.

The definition of rigid pair seems contrived. However they are abundant as the following result shows. Recall if $A$ is Henselian then $\CMS(A)$ is a KRS triangulated category. We show
\begin{theorem}
\label{rigid} Let $(A,\m)$ be a Henselian Gorenstein local ring and let $(B,\n)$ be another Gorenstein local ring (not necessarily Henselian). Let   $\Psi \colon \CMS(A) \rt \CMS(B)$ be a triangulated functor. Assume $\ker \Psi$  is a proper thick subcategory of $\CMS(A)$. Then $(\CMS(A), \ker \Psi)$ is a rigid pair.
\end{theorem}
Using Theorem \ref{rigid} we give abundant examples of rigid pairs $(\C, \D)$. However there are examples of rigid pairs which occur outside the context of Theorem \ref{rigid}. We prove
\begin{theorem}\label{rigid-ci}
Let $(A,\m)$ be a Henselian complete intersection local ring of codimension $c \geq 2$. For $i = 1, \ldots, c$ let $\CMS^{\leq i}(A)$ be the thick subcategory consisting of MCM modules of complexity $\leq i$. Then for $i = 2, \ldots,c$ we have that $(\CMS^{\leq i}(A), \CMS^{\leq i-1}(A))$ is a rigid pair.
\end{theorem}

\textbf{IX:}
There exists natural triangulated functors from $\CMS(A)$ to $\D(A)$. Here $\D(A)$ is the unbounded derived category of $A$-modules. These functors are constructed as follows.

Let $\apc(A)$ denote the
homotopy category of (unbounded) acyclic complexes of finitely  generated  free
$A$-modules
If $X \in \apc(A)$ then let $\Omega_0(X) = \coker(X^{-1} \rt X^0)$. We note that if $(A,\m)$ is a Gorenstein local ring then Buchweitz, \cite[4.4.1]{Bu},  constructs a triangle equivalence
$\Omega_0 \colon \apc(A) \rt \CMS(A)$. The inverse map $\Omega_0^{-1}$  takes $M$ to its complete resolution $\Fb_M$.  Let $X \in \CMS(A)$ be any MCM $A$-module. Let
$\Kc(A)$ denote the (unbounded) homotopy category of complexes of $A$-modules and let $q \colon \Kc(A) \rt \D(A)$ be the localization map.

Consider the functors
$$ \alpha_X^\prime \colon \apc(A) \rt \Kc(A) \ \quad \text{given by }  \ \mathbb{F} \rt \Fb\otimes X.$$
$$ \beta_X^\prime \colon \apc(A) \rt \Kc(A) \ \quad \text{given by }  \ \mathbb{F} \rt \Hom(\Fb,  X).$$
$$ \gamma_X^\prime \colon \apc(A) \rt \Kc(A) \ \quad \text{given by }  \ \mathbb{F} \rt \Hom(X, \Fb).$$
Finally consider the maps
$$\alpha_X = q \circ \alpha^\prime_X \circ \Omega_0^{-1}  \colon \CMS(A) \rt \D(A), $$
$$\beta_X = q \circ \beta^\prime_X \circ \Omega_0^{-1}  \colon \CMS(A) \rt \D(A), $$
$$\gamma_X = q \circ \gamma^\prime_X \circ \Omega_0^{-1}  \colon \CMS(A) \rt \D(A).$$
A natural question is to describe the kernels of these functors. We are unable to describe it for a general $X$. However when $X$ is free on the punctured spectrum of $A$ then we can prove:
\begin{theorem}
\label{tate} Let $(A,\m)$ be a Gorenstein local ring and let $X$ be an non-free MCM $A$-module which is free on the punctured spectrum of $A$. Then we have
\begin{enumerate}[\rm (1)]
\item
Assume $\dim A \geq 2$. The following assertions are equivalent:
\begin{enumerate}[\rm (i)]
\item
$M \in \ker \alpha_X$.
\item
$\Omega^n(M)\otimes X$ is an MCM $A$-module for all $n \in \Z$.
\item
$\depth \Omega^n(M)\otimes X  \geq 1$ for all $n \in \Z$.
\end{enumerate}
\item
Assume $\dim A \geq 3$. The following assertions are equivalent:
\begin{enumerate}[\rm (i)]
\item
$M \in \ker \beta_X$.
\item
$\Hom_A(\Omega^n(M),  X)$ is an MCM $A$-module for all $n \in \Z$.
\item
$\depth \Hom_A(\Omega^n(M),  X) \geq 3$ for all $n \in \Z$.
\end{enumerate}
\item
Assume $\dim A \geq 3$. The following assertions are equivalent:
\begin{enumerate}[\rm (i)]
\item
$M \in \ker \gamma_X$.
\item
$\Hom_A(X, \Omega^n(M))$ is an MCM $A$-module for all $n \in \Z$.
\item
$\depth \Hom_A(X, \Omega^n(M)) \geq 3$ for all $n \in \Z$.
\end{enumerate}
\end{enumerate}
\end{theorem}

The technique  to prove Theorem  \ref{tate} is also interesting. It proves the following result. Recall a ring $S$ has trivial Tor-vanishing if $\Tor^S_n(M. N) = 0$ for $n \gg 0$ implies that $M$ or $N$ has finite projective dimension. Analogously $S$ is said to have trivial Ext-vanishing if $\Ext^n_A(M, N) = 0$ for $n \gg 0$ implies that either $M$ has finite projective dimension or $N$ has finite injective dimension. If S is a Cohen-Macaulay local ring
having a canonical module, then both properties are equivalent, see \cite[3.2]{LM}.

\begin{theorem}
\label{tate-trivial} Let $(A,\m)$ be a Gorenstein local ring and let $X$ be an non-free MCM $A$-module which is free on the punctured spectrum of $A$. Assume $A$ has trivial Tor-vanishing property.  Let $M$ be an MCM $A$-module which is not free. Then we have
\begin{enumerate}[\rm (1)]
\item
Assume $\dim A \geq 2$. Then $\depth \Omega^n(M)\otimes X = 0$ for infinitely many $n > 0$.
\item
Assume $\dim A \geq 3$. Then
\begin{enumerate}[\rm (a)]
\item
$\depth \Hom_A(\Omega^n(M),  X) = 2$ for infinitely many  $n > 0$.
\item
$\depth \Hom_A(X, \Omega^n(M)) = 2$ for infinitely many  $n > 0$.
\end{enumerate}
\end{enumerate}
\end{theorem}

We now describe in brief the contents of this paper. In section two we discuss a few preliminaries that we need. In section three we discuss some properties of the Karoubi envelope of a triangulated category. In section four we prove Theorem \ref{first}. In the next section we prove Theorem \ref{second}. In section six we give a proof of Theorem \ref{third}. In the next section we give a proof of Theorem \ref{fourth}. In section eight we give a proof of Theorem \ref{fifth}. In the next section we prove Theorem \ref{equi}. In section ten we prove Theorem \ref{ri}. In the next section we prove Theorem \ref{ci}. In section twelve we prove Theorem's \ref{rigid} and \ref{rigid-ci}. In the next section we give proofs of Theorem \ref{tate} and \ref{tate-trivial}.
\section{Preliminaries}
In this section we discuss a few preliminary facts that we need.

\s Let $R$ be a commutative ring. Throughout we work with $R$-linear additive categories and $R$-linear functors between them.  For this notion see \cite[p.\ 28]{ARS}. We say an $R$-additive category $\C$ is Hom-finite if for any $X, Y$ in $\C$ the $R$-module $\Hom_\C(X, Y)$ has finite length.

\begin{remark}
Let $R$ be a $S$-algebra. If $\C$ is an additive $R$ category then it is also an additive $S$-category. Also if $\phi \colon \C \rt \D$ is an $R$-functor then it is also a $S$-functor.
\end{remark}

\s Let $\C$ be a $R$-category. A map $e \colon M \rt M$ is said to be an \emph{idempotent}
if $e^2 = e$. We say an idempotent $e$ splits provided there
is a factorization $M \xrightarrow{p} K \xrightarrow{u} M$ such that $e = u\circ p$ and $p\circ u = 1_K$. We say that $\C$ is \emph{idempotent} complete if every idempotent in $\C$ splits.

\s Let $\C$ be any $R$-category. The \emph{Karoubi envelope}  of $\C$ is an idempotent complete $R$-category $K(\C)$ with a $R$-functor $\theta_\C \colon \C \rt K(\C)$ such that if $f \colon \C \rt \D$ is a $R$-functor where $\D$ is an idempotent complete $R$-category then there exists a (essentially unique) $R$-functor $K(f) \colon K(\C) \rt \D$ such that
$f = K(f)\circ \theta_\C$. Furthermore the pair $(\theta_\C, K(\C))$ is unique upto equivalence of $R$-categories, see \cite[I.6.10]{Ka}. We will need two facts. The map $\theta_\C$ is fully faithful and for any  $Y \in K(\C)$ there exists $X \in \C$ such that
$Y$ is a direct summand of $\theta_\C(X)$.

For the definition of Krull-Remak-Schmidt(KRS) category see \cite[1.1.3]{LW}.
We need the following easily proven fact.
\begin{proposition}\label{K-hom-f}
If $\C$ is a Hom-finite $R$-category then
\begin{enumerate}[\rm (1)]
\item
$K(\C)$ is a Hom-finite $R$-category.
\item
$K(\C)$ is a KRS category.
\end{enumerate}
\end{proposition}

\s Let $\C$ be a triangulated category. Then there exists a unique triangulated structure on $K(\C)$ such that $\theta_\C$ is a triangulated functor. Furthermore if $f \colon \C \rt \D$ is a $R$- triangulated functor where $\D$ is an idempotent complete triangulated $R$-category then there exists a (essentially unique) triangulated  $R$-functor \\ $K(f) \colon K(\C) \rt \D$ such that
$f = K(f)\circ \theta_\C$. For these facts see \cite[1.5]{BS}.

\s \label{center-eta} Let $(A,\m)$ be a Gorenstein local ring and let $M$ be a MCM $A$-module. There is a natural ring homomorphism $\eta_M \colon A \rt \sHom_A(M,M)$ defined by $q \rt \mu_a $ where $\mu_a \colon M \rt M$ defined by $\mu_a(m) = am$. We note that $\eta_M(a) \in Z_M = $ center of $\sHom_A(M,M)$. Furthermore via $\eta_M$ we get that $\sHom_A(M,M)$ (and so $Z_M$) is a finite $A$-module.

\s Let $(A, \m)$ be a Gorenstein local ring. Then $\CMS^0(A)$ denotes the thick subcategory of MCM $A$-modules $M$ such that $M_P$ is free for all primes  $P \neq \m$.
\s Let $(A,\m)$ be a local Noetherian ring. Let $I \subseteq \m$ be an ideal in $A$.
Let $M$ be a finitely generated $A$-module.
Recall an $I$-\textit{filtration} $\Fc = {\{M_n\}}_{n \in \Z}$ on $M$ is a
collection of submodules of $M$ with the properties
\begin{enumerate}
\item
$M_n \supseteq M_{n+1}$ for all $n \in \Z$,
\item
$M_n = M$ for all $n \ll 0$,
\item
$I M_n \subseteq M_{n+1}$ for all $n \in \Z$.
\end{enumerate}
If
$I M_n = M_{n+1}$ for all $n \gg 0$ then we say $\Fc$ is $I$-\emph{stable}.

Let $\eR_A(I) = \bigoplus_{n\in \Z} I^n$ be the extended Rees algebra of $A$ \wrt \ $I$.
If $\Fc = \{M_n\}_{n \in \Z}$ is an $I$-stable filtration on $M$, then  set
$\eR(\Fc ,M) = \bigoplus_{n \in \Z}M_n$ the \emph{extended Rees-module} of $M$
\emph{\wrt }\ $\Fc$.
 Notice that $\eR(\Fc,M)$ is a finitely generated graded $\eR_A(I)$-module. If $\Fc$ is $I$-adic then set $\eR(\Fc,M) = \eR(I, M)$.

\section{The Karoubi envelope}

Let $\C, \D$ be triangulated $R$-categories.
\s
Recall, a triangulated $R$-functor $F\colon \C\rt \D$ is called an \emph{equivalence up to direct summand's}
if it is fully faithful and any object $X \in \D$  is isomorphic to a direct summand of $F(Y)$ for
some $Y \in \C$.

We will need the following:
\begin{lemma}
\label{krs-equi}
Let  $\C, \D$ be triangulated KRS $R$-categories. If $f \colon \C \rt \D$ is a $R$-triangulated functor which is an equivalence upto direct summands then $f$ is dense. In particular $f$ is an equivalence.
\end{lemma}
\begin{proof}
Let $Y \in \D$. Then $Y$ is a direct summand of $f(Z)$ for some $Z \in \C$. So there exists
an idempotent $\phi \colon f(Z) \rt f(Z)$ and a factorization $f(Z) \xrightarrow{p} Y \xrightarrow{u} f(Z)$ such that $\phi = u\circ p$ and $p\circ u = 1_Y$. As $f$ is fully faithful there exists $\psi \colon Z \rt Z$ such that $\phi = f(\psi)$. We note that $\psi$ is also an idempotent in $\C$. As $\C$ is KRS we get that $\psi$-splits. Let $Z \xrightarrow{q} K \xrightarrow{v} Z$ such that $\psi = v\circ q$ and $q\circ v = 1_K$.
It follows that $\phi$ also splits as $f(Z) \xrightarrow{f(q)} f(K) \xrightarrow{f(v)} f(Z)$ such that $\phi = f(v)\circ f(q)$ and $f(q)\circ f(v) = 1_f(K)$. By uniqueness of splittings
we get $Y \cong f(K)$. So $f$ is dense.
\end{proof}

As an application of Lemma \ref{krs-equi} we give a considerably simpler proof of Elkik's result \cite{E}  at least for MCM modules over Henselian Gorenstein rings.
\begin{corollary}
\label{elkik} Let $(A,\m)$be a Henselian  Gorenstein local ring. Let $M$ be a maximal \CM \ $\wh{A}$-module which is free on $\Spec^*(\wh{A})$. Then there exists maximal \CM \ $A$-module $N$ which is free on the punctured spectrum of $A$ such that $M \cong \wh{N}$.
\end{corollary}
\begin{proof}
We have nothing to show if $M$ is free. So assume $M$ is not free. By \cite[3.2]{T}, the map $\eta \colon \CMS^0(A) \rt \CMS^0(\wh{A})$ is an equivalence upto direct summands. As both $A, \wh{A}$ are Henselian both  $\CMS^0(A)$ and $ \CMS^0(\wh{A})$ are KRS categories. So by  \ref{krs-equi} we get that $\eta$ is dense. Thus there exists $E \in \CMS(A)$ such that $\wh{E} \cong M$ in $\CMS(\wh{A})$. So there exists free $\wh{A}$-modules $F, G$-modules such that $F \oplus M \cong G \oplus \wh{E}$ as $\wh{A}$-modules. We may assume that $E$ has no free summands. So $\wh{E}$  has no free summands. As $\wh{A}$ is a KRS category it follows that   $M \cong H\oplus \wh{E}$ for some free $\wh{A}$-module $H$. Say $H = \wh{A}^r$. Set
$N = A^r \oplus E$. Then $\wh{N} \cong M$. The result follows.
\end{proof}

\s If $\C$ is an additive category then let $K(\C)$ denote the Karoubi envelope of $\C$.
The next result is a crucial ingredient in the proof of  Theorem \ref{first}.

\begin{theorem}
\label{crucial}
Let $\C, \D$ be Hom-finite triangulated $R$-categories with $\D$ a KRS category. Suppose there exists $\eta \colon \C \rt \D$ an equivalence up to direct summands. Then  $K(\eta) \colon K(\C) \rt \D$ is an equivalence of triangulated $R$-categories.
\end{theorem}
\begin{proof}
By Proposition \ref{K-hom-f} and   Lemma \ref{krs-equi} it suffices to prove that $K(\eta)$ is an equivalence up to direct summands. Let $\theta \colon \C \rt K(\C)$ be the  canonical functor.

We first show that $K(\eta)$ is  fully faithful. Let $U, V \in K(\C)$. We want to show
the map $K(\C)_{U,V} \colon \Hom_{K(\C)}(U, V) \rt \Hom_\D(K(\eta)(U), K(\eta)(V))$ is an isomorphism of $R$-modules.

There exists  $X, Y \in \C$ such that $U$ is a direct summand of $\theta(X)$ and $V$ is a direct summand of $\theta(Y)$. So there exists $U_1,V_1$ in $K(\C)$ with $U\oplus U_1 = \theta(X)$ and
$V \oplus V_1 = \theta(Y)$.

We have a commutative diagram
\[
\xymatrix{
\
&\Hom_{K(\C)}(\theta(X), \theta(Y))
\ar@{->}[dr]^{K(\eta)(\theta(X), \theta(Y))}
 \\
\Hom_\C(X, Y)
\ar@{->}[rr]_{\eta(X,Y)}
\ar@{->}[ur]^{\theta(X,Y)}
&\
&\Hom_\D(\eta(X), \eta(Y))
}
\]

We note that $K(\eta)\circ \theta(X) = \eta(X)$. Similar assertion for $Y$. We note that $\eta(X,Y)$ and  $\theta(X,Y)$ are isomorphisms. So $K(\eta)(\theta(X), \theta(Y))$ is an isomorphism.

We have a commutative diagram
\[
  \xymatrix
{
 \Hom_{K(C)}(U, V)
\ar@{->}[r]^{ K(\C)_{U,V} }
\ar@{->}[d]^{i}
 & \Hom_\D(K(\eta)(U), K(\eta)(V))
\ar@{->}[d]^{j}
\\
 \Hom_{K(C)}(\theta(X), \theta(Y))
 \ar@{->}[r]^{ K(\eta)(\theta(X), \theta(Y))  }
 & \Hom_\D(\eta(X), \eta(Y))
\
 }
\]
Here $i,j$ are inclusions. Also $K(\eta)(\theta(X), \theta(Y))$ is an isomorphism. So $K(\C)_{U,V}$ is an inclusion of $R$-modules. In particular

\begin{equation*}
\ell_R(\Hom_{K(\C)}(U,V)) \leq
\ell_R (\Hom_\D(K(\eta)(U), K(\eta)(V))).  \tag{i}
\end{equation*}
Similarly we have
\begin{equation*}
\ell_R(\Hom_{K(\C)}(U,V_1)) \leq
\ell_R (\Hom_\D(K(\eta)(U), K(\eta)(V_1))),  \tag{ii}
\end{equation*}
\begin{equation*}
\ell_R(\Hom_{K(\C)}(U_1,V)) \leq
\ell_R (\Hom_\D(K(\eta)(U_1), K(\eta)(V))).  \tag{iii}
\end{equation*}
Finally we have
\begin{equation*}
\ell_R(\Hom_{K(\C)}(U_1,V_1)) \leq
\ell_R (\Hom_\D(K(\eta)(U_1), K(\eta)(V_1))).  \tag{iv}
\end{equation*}
If we add the left side of the above four inequalities we get \\ $\ell_R(\Hom_{K(\C)}(\theta(X), \theta(Y)))$. While if add up the right side of the four inequalities we obtain
$$\ell_R(\Hom_\D(K(\eta)(\theta(X), K(\eta)(\theta(Y))) = \ell_R(\Hom_\D(\eta(X), \eta(Y))).$$
We had earlier demonstrated that the natural map \\ $\Hom_{K(\C)}(\theta(X), \theta(Y))) \rt \Hom_\D(\eta(X), \eta(Y))$ is an isomorphism of $R$-modules. In particular they have the same length.
It follows that all the four inequalities are equalities. In particular $K(\C)_{U, V}$ is an isomorphism (as we had earlier proved it is an injection). Thus $K(\eta)$ is fully faithful.

Let $V \in \D$. Then by assumption $V$ is a direct summand of $\eta(U)$ for some $U \in \C$. We note that $K(\eta)(\theta (U)) = \eta(U)$. Thus $K(\eta)$ is an equivalence upto direct summands.
\end{proof}

\section{Proof of Theorem \ref{first}}
In this section we give a proof of Theorem \ref{first}. We first need
the following result:
\begin{lemma}
\label{A-K} Let $(A,\m)$ be a Gorenstein local ring. Then there is an equivalence of $A$-triangulated categories $K(\CMS^0(A)) \rt \CMS^0(\wh{A})$.
\end{lemma}
\begin{proof}
  We note that the $A$-functor $\eta \colon \CMS^0(A) \rt \CMS^0(\wh{A})$ is an equivalence upto direct summands, see \cite[3.2]{T}. Furthermore $\CMS^0(A)$ and $\CMS^0(\wh{A})$ are Hom-finite $A$-categories. We also have that $\CMS^0(\wh{A})$ is a KRS category. The result follows from Theorem \ref{crucial}.
\end{proof}
We now give
\begin{proof}[Proof of Theorem \ref{first}]
By H. Matsui, \cite[4.6]{MH} we have an equivalence \\ $\Phi^0 \colon \CMS^0(A) \rt \CMS^0(B)$. It follows that $K(\CMS^0(A)) \cong K(\CMS^0(B))$. By Lemma \ref{A-K} it follows that $\CMS^0(\wh{A}) \cong \CMS^0(\wh{B})$. We note that $\wh{A}, \wh{B}$ are Henselian. The result follows from  \cite[6.12]{P-stable}.
\end{proof}

\section{Proof of Theorem \ref{second}}
We need a few preliminaries before we  prove Theorem \ref{second}.

\s\label{setup-second} Let $(A,\m)$ be a Gorenstein local ring of dimension $d$. Let $M \in \CMS(A)$. Consider $T_M = \sHom_A(M, M)$.  Let $Z_M = $ center of $T$.  We have a ring homomorphism $\rho \colon A \rt Z_M$. We note that via this map $Z$ is a finite $A$-module. Let $N \in \CMS^0(A)$. We have $\Ext^n_A(M, N)$ has finite length as an $A$-module for $n \geq 1$. We note that $\Ext^n_A(M, N) = \sHom_A(M, \Omega^{-n}(N))$ and so is a  (right) $T_M$ and hence a  $Z_M$-module. If $u \in \Ext^n_A(M, N)$ then note that for $a \in A$ we have $a.u = \rho(a).u$
We define
\begin{enumerate}
  \item  $\curv_A(M, N) = \limsup_n \sqrt[n]{\ell_A(\Ext^n_A(M, N)}$
  \item $\curv_{Z_M}(M, N) = \limsup_n \sqrt[n]{\ell_{Z_M}(\Ext^n_A(M, N)}$
  \item $\cx_A(M, N)  = \inf\{ r \in \mathbb{N} \mid  \limsup_n \frac{\ell_A(\Ext^n(M, N))}{n^{r - 1}}  < \infty \}.$
  \item $\cx_{Z_M}(M, N)  = \inf\{ r \in \mathbb{N} \mid  \limsup_n \frac{\ell_{Z_M}(\Ext^n(M, N))}{n^{r - 1}}  < \infty \}.$
\end{enumerate}
Let $k$ be the residue field of $A$. Let $0 \rt Y \rt X \rt k \rt 0$ be an MCM approximation of $k$, i.e., $X$ is MCM and $\projdim_A Y < \infty$. We prove:
\begin{theorem}
\label{growth} (with hypotheses as in \ref{setup-second}). We have
\begin{enumerate}[\rm (1)]
  \item $\curv_A(M, N) \leq \curv_A(M,X) = \curv_A(M)$.
  \item $\cx_A(M, N) \leq \cx_A(M, X) = \cx_A(M)$.
\end{enumerate}
\end{theorem}
Before proving Theorem \ref{growth} we need a few preliminaries.  We note that $X \cong \Omega^{-d}(\Omega^d(k))$ in $\CMS(A)$. By a result of Takahashi, see \cite[4.3]{T-p}, we get that \\
$\thick(X) = \CMS^0(A)$. As $Y$ has finite injective dimension it follows that for $n > d$ we have  $\ell(\Ext^n_A(M, X)) = \ell (\Ext^n_A(M, k)) = \beta_n(M)$. So $\curv_A(M, X) = \curv_A(M)$.
\s\label{t-r} Let $\Tc$ be a triangulated category with shift functor $\sum$. Let $\Ic_1, \Ic_2$ be two subcategories of $\Tc$. We denote by $\Ic_1 * \Ic_2$ the full subcategory  of $\Tc$ consisting of objects $M$  such that there is a  triangle  $M_1 \rt M \rt M_2 \rt \sum M_1$ with $M_i \in \Ic_i$.
Let $\Ic$ be a full subcategory of $\Tc$. By $\langle \Ic  \rangle$ we denote the smallest subcategory of $\Tc$ containing $\Ic$  which is closed under finite direct sums, direct summands, shifts and isomorphisms. Set  $\Ic_1 \diamond \Ic_2 = \langle  \Ic_1 * \Ic_2 \rangle$.

Set $\langle \Ic \rangle_0 = 0$. Then inductively define $\langle   \Ic   \rangle_i = \langle \Ic \rangle_{i-1} \diamond  \langle \Ic  \rangle.$
It is clear that $\langle \Ic \rangle_{i-1} $ is a subcategory of  $\langle   \Ic   \rangle_i$ for all $i \geq 1$.
Set
$$ \langle \Ic\rangle_{\infty} = \bigcup_{i \geq 0} \langle   \Ic   \rangle_i.$$

\s \label{thick} Let $\Ic$ be a subcategory of $\Tc$. By $\thick(\Ic)$ we mean the intersection of all triangulated subcategories of $\Tc$ containing $\Ic$. It can be shown that
$ \thick(\Ic) =  \langle \Ic\rangle_{\infty}$.
\begin{proof}[Proof of Theorem \ref{growth}]
Let $\Ic = \text{add}(X)$. We have $$  \langle \Ic\rangle_{\infty}  = \thick(\text{add}(X)) = \CMS^0(A).$$

(1) We make the observation that
\begin{enumerate}[\rm (i)]
\item $\curv_A(M, \Omega^j(N)) = \curv_A(M, N)$ for all $j \in \Z$.
\item If $N_1$ is a direct summand of $N$ then
$\curv_A(M, N_1) \leq \curv_A(M, N)$.
\item
$\curv_A(M, N_1 \oplus N_2) \leq \max\{ \curv_A(M, N_1), \curv_A(M, N_2) \}$.
\end{enumerate}
It follows that if $\curv_A(M, E) \leq \curv_A(M, X)$ for all $E \in \C$ (here $\C$ is a subcategory of $\CMS^0(A)$) then
$\curv_A(M, N) \leq \curv_A(M, X)$ for all $N \in \langle \C  \rangle$.

We prove by induction on $i$ that if $N \in \langle   \Ic   \rangle_i $ then $\curv_A(M,N) \leq \curv(M, X)$.  By the above assertions we have nothing to prove if $i = 1$. Assume the result for $i = r$ and we prove it for $i = r +1$. Let $N \in \langle   \Ic   \rangle_r *  \langle   \Ic   \rangle $. Then by the structure of triangles in $\CMS(A)$ we have an exact sequence $0 \rt E_r \rt N \oplus F \rt E_1 \rt 0$ where $E_i \in \langle   \Ic   \rangle_i$ and $F$ a free $A$-module.
So for $n \geq 1$ we have
\[
\ell(\Ext^n(M, N)) \leq \ell(\Ext^n(M, E_r)) +  \ell(\Ext^n(M, E_1)).
\]
Set $\alpha = \curv(M) = \curv(M, X)$. Let $\epsilon > 0$. By induction hypotheses there exists positive integer $n_0(\epsilon)$  such that for $n \geq n_0(\epsilon)$ we have
\[
\ell(\Ext^n(M, E_r))  \leq (\alpha + \epsilon)^n \quad \text{and} \quad \ell(\Ext^n(M, E_1))  \leq (\alpha + \epsilon)^n.
\]
It follows that $\curv(M, N) \leq \alpha + \epsilon$. As $\epsilon > 0$ was arbitrary it follows that $\curv(M, N) \leq \alpha$. So our assertion holds for
$N \in \langle   \Ic   \rangle_r *  \langle   \Ic   \rangle $. By our earlier argument it follows that $\curv(M, L) \leq \alpha$ for all $L \in \langle   \Ic   \rangle_{r +1}$. Thus the result holds by induction.

(2) This is similar to (1).
\end{proof}
We need the following result.
\begin{proposition}\label{bella}
Let $(A,\m)$ be a local Noetherian ring and let $B$ be a commutative $A$-algebra such that $B$ is a finite $A$-module. Let $k = A/\m$. Let $\n_1, \ldots,\n_r$ be the maximal ideals of $B$ and set $\delta = \max \{ \dim_k B/\n_i \mid 1 \leq i \leq r\}$. If $D$ is a finite length $B$-module then
\[
\ell_A(D) \geq \ell_B(D) \geq \frac{1}{\delta}\ell_A(D).
\]
\end{proposition}
\begin{proof}
  After going mod the annihilator of $D$ we may assume that $A$ and $B$ are Artin. So $B = B_1\times \cdots \times B_s$ where $B_i$ is local with maximal ideal $\n_i$. Let $D = D_1\times \cdots \times D_s$ along this decomposition.
  Then we have $\ell_B(D) = \sum_{i =1}^{s}\ell_{B_i}(D_i)$ and
  \[
  \ell_A(D) = \sum_{i = 1}^{s} \ell_A(D_i) = \sum_{i=1}^{s} \ell_{B_i}(D_i)\dim_k B/\n_i.
  \]
  The result follows.
\end{proof}
Next we show
\begin{proposition}\label{compare}
(with hypotheses as in \ref{setup-second}). Let $M \in \CMS(A)$ and $N \in \CMS^0(A)$. We have
\begin{enumerate}[\rm (1)]
  \item $\curv_A(M, N) = \curv_{Z_M}(M, N)$.
  \item $\cx_A(M, N)$ is infinite if and only if $\cx_{Z_M}(M, N)$ is infinite.
  \item If $\cx_A(M, N)$ is finite  then $\cx_{Z_M}(M, N) = \cx_A(M, N)$.
\end{enumerate}
\end{proposition}
\begin{proof}
Let $\n_1, \ldots, \n_r$ be all the maximal ideals in $Z_M$. Set $\delta = \max \{ \dim_k B/\n_i \mid 1 \leq i \leq r\}$. If $D$ is a finite length $B$-module then by \ref{bella} we have
\[
\ell_A(D) \geq \ell_{Z_M}(D) \geq \frac{1}{\delta}\ell_A(D).
\]
Set $D = \Ext^n_A(M, N)$ and conclude the assertions in the proposition.
\end{proof}
We now give
\begin{proof}[Proof of Theorem \ref{second}]
let $d = \dim A$ and $r = \dim B$.
(1) Let $\Phi(M) = D$. Then $T = \sHom_A(M, M) \cong \sHom_B(D, D)$. So they have isomorphic centers. If $N \in \CMS^0(A)$ then $\Phi(N) \in \CMS^0(B)$ (and conversely). It follows that
$\curv_{Z_M}(M, N) = \curv_{Z_D}(D, \Phi(N))$. Taking $N = \Omega^{-d}(\Omega^d(k))$ we obtain \\ $\curv_A(M) =\curv_{Z(D)}(D, \Phi(N))$. It follows from \ref{growth} and \ref{compare}  that
 $\curv_A(M) \leq \curv_B(D)$. By considering $\Phi^{-1}$ we get that $\curv_B(D) \leq \curv_A(M)$. Thus \\ $\curv_A(M) = \curv_B(D)$.

 (2) By (1) the sets $\{ \curv_A M \mid M \in \CMa(A) \}$ and  $\{ \curv_A D \mid D \in \CMa(B) \}$ are equal. The maximum of the first set is $\curv_A(\Omega^{\dim A}(A/\m)) = \curv_A(A/\m)$ and that of second is  $\curv_B(\Omega^{\dim B}(B/\n)) = \curv_B(B/\n)$. The result follows.

 (3) and (4) This follows as in (1). We have to use \ref{growth} and \ref{compare}.
\end{proof}
\section{Proof of Theorem \ref{third}}
In this section we give a proof of Theorem \ref{third}. We first need the following result.
\begin{lemma}
  Let $(R,\m)$ be a complete Gorenstein local ring. Let $M$ be a MCM indecomposable $R$-module and  assume $M \ncong R$. Then
  \begin{enumerate}[\rm (1)]
    \item $D = \sHom_R(M,M)/\rad \sHom_R(M,M)$ is a division ring.
    \item $D$ is a finite extension of $k = R/\m$.
    \item The center of $D$ is a finite field extension of $k$.
  \end{enumerate}
\end{lemma}
\begin{proof}
(1) As $M$ is indecomposable we have that $\Hom_R(M,M)$  is a local (not necessarily commutative) ring, (this follows from \cite[21.35]{Lam}). The ideal of $\Hom_R(M,M)$ consisting of all maps which factor through a free $R$-module is clearly contained in the radical of $\Hom_R(M,M)$, see \cite[2.2]{P-ar}. The result follows.

(2) We have a ring map $\eta \colon A \rt \Hom_A(M, M)$ with $a \mapsto \mu_a$ where $\mu_a \colon M \rt M$ is multiplication by $a$. We note that if $a \in \m$ then $\mu_a(M) \subseteq \m M$.
 So $\mu_a \in \rad \Hom_A(M, M)$, see \cite[2.2]{P-ar}. Next $\eta^\prime \colon A \rt \sHom_A(M, M)$ obtained by composing $\eta$ with the projection $\Hom_A(M, M) \rt \sHom_A(M, M)$. The image of $\mu_a$ in $\sHom_A(M.M)$ is in its radical. So we have a map $k \rt D$. As $\sHom_A(M,M)$ is a finite  $A$-module, it follows that $D$ is a finite extension of $k$.

(3) This follows from (2).
\end{proof}
We now give
\begin{proof}[Proof of Theorem \ref{third}]
By proof of Theorem \ref{first} the equivalence $\Phi \colon \CMS(A) \cong \CMS(B)$ induces an equivalence $\Psi \colon \CMS^0(\wh{A}) \cong \CMS^0(\wh{B})$. Let $M \in \CMS^0(\wh{A})$ be indecomposable.  Let $N = \Psi(M)$. We note that $\sHom_A(M, M) \cong \sHom_B(N,N)$. This implies
$$D_M = \frac{\sHom_A(M, M)}{\rad \sHom_A(M,M)} \cong \frac{\sHom_B(N, N)}{\rad \sHom_B(M,M)} = D_N.$$
So it follows that $Z_M$, the center of $D_M$, is isomorphic to $Z_N$, the center of $D_N$.
We have $Z_M$ is a finite extension of $k$ and $Z_N$ is a finite extension of $l$.

(1) We have $\charp k = \charp Z_M = \charp Z_N = \charp l$.

(2) $Z_M$ is a finite extension of $k$. So $k$ is finite if and only if $Z_M$ is finite. Similarly $l$ is finite if and only if $Z_N$ is finite. As $Z_M \cong Z_N$ the result follows.

(3) Assume $k$ is infinite. As $Z_M$ is a finite extension of $k$ we get that $|Z_M| = |k|$. As $l$ is also infinite, similarly we get that $|Z_N| = |l|$. Since $Z_M \cong Z_N$ the result follows.

(4) If $k$ is perfect then $Z_M$ being a finite extension of $k$ is also perfect. So $Z_N \cong Z_M$ is also perfect. As $Z_N$ is a finite extension of $l$ it follows that $l$ is also perfect
 (this face is known though a little tricky to prove, see \cite{WS}).

(5) If $k$ is algebraically closed then $Z_M = k$. So $Z_N \cong k$ is algebraically closed. We have that $Z_N$ is a finite extension of $l$. If $Z_N \neq l$ then necessarily $\charp l = 0$ and $l$ is a real closed field with $l(\sqrt{-1}) = Z_N \cong k$; see \cite[VI, 9.3]{L} and \cite[XI, 2.4]{L}. The result follows.
\end{proof}

\section{Proof of Theorem \ref{fourth}}
\s \label{set-4} Let $I \subseteq \m$ be an ideal in $A$ and let  $\eR_A(I) = A[It, t^{-1}]$.
If $\eR_A(I)$ is Gorenstein then let $\CMS_\Z(\eR_A(I))$ denotes the category of $\Z$-graded MCM $\eR_A(I)$-modules.
We first construct a triangulated functor $\Psi \colon \CMS_\Z(\eR_A(I)) \rt \CMS(A)$.

\s Let $S$ be a Noetherian ring. Let  $\apc(S)$ denotes the
homotopy category of (unbounded) acyclic complexes of finitely  generated projective
$S$-modules.
It is a full
subcategory of homotopy category of complexes of $S$-modules ($K(S)$), closed under translation and forming mapping cones, hence
inherits a triangulated structure from $K(S)$.  If $T = \bigoplus_{n \in \Z}T_n$ is a graded Noetherian ring then  $\gapc(T)$ denotes the
homotopy category of (unbounded) acyclic complexes of finitely  generated graded projective
$T$-modules (and degree zero maps).
If $X \in \apc(S)$ then let $\Omega_0(X) = \coker(X^{-1} \rt X^0)$. We note that if $(A,\m)$ is a Gorenstein local ring then Buchweitz, \cite[4.4.1]{Bu} constructs a triangle equivalence
$\Omega_0 \colon \apc(A) \rt \CMS(A)$. Similarly if $T$ is a graded Gorenstein ring  then a similar proof to  \cite[4.4.1]{Bu} yields an equivalence
$\Omega_0 \colon \gapc(T) \rt \CMS_\Z(T)$ where $\CMS_\Z(T)$ is the stable category of all MCM graded $T$-modules.

\begin{construction}
\label{mom} We give a triangulated functor $\Phi \colon \gapc(\eR_A(I)) \rt \apc(A)$ as follows.
We note that $\eR_A(I), A[t,t^{-1}]$ are $^*$-local rings. So graded projective modules over them are free. Furthermore $\eR_A(I)_{t^{-1}} = A[t,t^{-1}]$.
First let $\mathbf{f} \colon \gapc(\eR_A(I)) \rt \gapc(A[t,t^{-1}])$ given by $X \mapsto X_{t^{-1}}$. Clearly $\mathbf{f}$ is a triangulated functor. Let $\mathbf{g} \colon \gapc(A[t,t^{-1}]) \rt \apc(A)$ defined by
$\mathbf{g}(X) = X_0$ (here $X_0$ is the complex of free $A$-modules consisting of degree zero components of $X$). Then it is evident that $\mathbf{g}$ is a triangulated functor. Set $\Phi = \mathbf{g}\circ \mathbf{f}$. So $\Phi \colon \gapc(\eR_A(I)) \rt \apc(A)$ is a triangulated functor.
\end{construction}

\begin{remark}\label{induce}
  The triangulated functor $\Phi \colon \gapc(\eR_A(I)) \rt \apc(A)$ yields a triangulated functor $\Psi \colon \CMS_\Z(\eR_A(I)) \rt \CMS(A)$.  The rest of this section  involves understanding $\Psi$.
\end{remark}

We first prove:
\begin{lemma}
\label{bella-l} Let $(A,\m)$ be a \CM \ local ring and let $I$ be an ideal in $A$ with $\eR_A(I)$ \CM. Let $E$ be a MCM $\eR_A(I)$-module. Then there exists a MCM $A$-module $M$ and an $I$-stable filtration $\Fc$ on $M$ such that $\eR(\Fc, M) \cong  E$.
\end{lemma}
\begin{proof}
Set $\eR = \eR_I(A)$.
We note that $t^{-1}$ is $\eR$-regular and as $E$ is a MCM $\eR$-module we get that $t^{-1}$ is $E$-regular.
By \cite[3.1]{P-Itoh-g} there exists a finitely generated $A$-module $M$ and an $I$-stable filtration $\Fc$ on $M$ such that $\eR(\Fc, M) \cong  E$.
Let $P$ be a minimal rime of $E$. Then $P$ is a minimal prime of $\eR$. So $P = \q A[t, t^{-1}] \cap \eR$ for some minimal prime $\q $ of $A$. Let $d = \dim A$. Let $\q = \q_0 \subseteq \q_1 \subseteq \cdots \subseteq \q_d = \m$ be the chain of prime ideals in $A$. Then we have a chain $P= P_0 \subseteq \cdots \subseteq P_d = \m A[t,t^{-1} \cap \eR$  where $P_i = \q_i A[t,t^{-1}] \cap \eR$ in the  support of $E$. It follows that the $^*$-dimension of $E_{t^{-1}} = M[t,t^{-1}]$ is $d$ as a  graded $A[t,t^{-1}]$-module.  We have $U = M[t,t^{-1}]$ is \CM \ as a $A[t,t^{-1}]$-module. Note $\n = \m A[t,t^{-1}]$ is
 the $^*$-maximal ideal of $A[t,t^{-1}]$. Let $\kappa(\n)$ be the residue field of $A[t,t^{-1}]_\n$. As $U_\n$ is a  \CM  \  $A[t,t^{-1}]_\n$-module of dimension $d$ we have
 $$\Ext^i_{A[t,t^{-1}]_\n}(\kappa(\n), U_\m) = 0 \quad \text{for} \ i < d. $$
 Notice that $\Ext^i_{A[t,t^{-1}]}(A/\m[t,t^{-1}], U)$ is a graded $A[t,t^{-1}]$-module whose localization at its $^*$-maximal ideal $\n$ is zero for $i < d$. It follows that \\
  $\Ext^i_{A[t,t^{-1}]}(A/\m[t,t^{-1}], U) = 0$ for $i < d$, see \cite[1.5.15]{BH}. We have
  $$\Ext^i_A(A/\m, M)\otimes_A A[t,t^{-1}]  \cong \Ext^i_{A[t,t^{-1}]}(A/\m[t,t^{-1}], U) = 0 \quad \text{for} \ i < d.$$
  As the map $A \rt A[t,t^{-1}]$ is faithfully flat it follows that $\Ext^i_A(A/\m, M) = 0$ for $i < d$. So $M$ is a MCM $A$-module.
\end{proof}

\s \label{infty} Let $\Fc$ and $\Gc$ be $I$-stable filtration's on $M$ and $N$ respectively. Let \\ $f \colon \eR(\Fc, M) \rt \eR(\Gc, N)$ be $\eR_A(I)$-linear. Then there exists $A$-linear maps
\\ $f_n \colon \Fc_n \rt \Gc_n$ for all $n \in \Z$ with $f(mt^n) = f_n(m)t^n$. Let $s(\Fc, \Gc) = \max \{ i \mid \Fc_i = M \ \text{and} \ \Gc_i = N \}$. We note that $\Fc_i = M$ and $\Gc_i = N$ for all $i \leq s(\Fc, \Gc)$. We have
\begin{lemma}\label{f*}
(with hypotheses as in \ref{infty}). If $n, m \leq s(\Fc, \Gc)$ then $f_n = f_m \colon M \rt N$.
\end{lemma}
\begin{proof}
  It suffices to prove $f_n = f_{n-1}$ for all $n \leq s(\Fc, \Gc)$.
  We have
  \begin{align*}
    f_{n-1}(m)t^{n-1} &= f(m t^{n-1}) = f(t^{-1}mt^{n})   \\
     & = t^{-1}f(mt^{n}) = t^{-1}f_n(m)t^{n} = f_n(m)t^{n-1}.
  \end{align*}

  The result follows.
\end{proof}
\begin{definition}
(with hypotheses as in \ref{infty}). Define $f_* \colon M \rt N$ to be $f_n$ for (any) $n \leq s(\Fc, \Gc)$.
\end{definition}
The next result describes $\Psi$.
\begin{theorem}
\label{psi}(with hypotheses as in \ref{set-4}). The functor $\Psi \colon \CMS_\Z(\eR_A(I)) \rt \CMS(A)$ induced by $\Phi$ takes $\Psi(\eR(\Fc, M)) = M$ and
$$\Psi(f \colon \eR(\Fc, M) \rt \eR(\Gc, N)) = f_* \colon M \rt N.$$
\end{theorem}
We will need a few preliminary results.
\s \label{inv} Let $f \colon M[t,t^{-1}] \rt N[t,t^{-1}]$ be a graded  $A[t,t^{-1}]$-linear map. Define \\ $f_n \colon M \rt N$ by setting $f(mt^n) = f_n(m)t^n$ for $n \in \Z$.
We show
\begin{lemma}\label{bella-3}(with hypotheses as in \ref{inv}). We have $f_n = f_{n-1}$ for all $n \in \Z$.
\end{lemma}
\begin{proof}
  We have
  $$f_{n-1}(m)t^{n-1} = f(mt^{n-1}) = f(t^{-1}mt^{n}) = t^{-1}f_n(m)t^{n} = f_n(m)t^{n-1}.$$
  The result follows.
\end{proof}
\begin{definition}
Let $f \colon M[t.t^{-1}] \rt N[t,t^{-1}]$ be $A[t,t^{-1}]$-linear. Define $\wt{f} \colon M \rt N$ to be (any) $f_n$ with $n \in \Z$.
\end{definition}

\begin{lemma}
\label{loc-tilde}
Let $\Fc$ and $\Gc$ be $I$-stable filtration's on $M$ and $N$ respectively. Let \\ $f \colon \eR(\Fc, M) \rt \eR(\Gc, N)$ be $\eR_A(I)$-linear. Consider the map $f_{t^{-1}} \colon \eR(\Fc, M)_{t^{-1}} \rt \eR(\Gc, N)_{t^{-1}}$. We have $\eR(\Fc, M)_{t^{-1}} = M[t,t^{-1}]$ and  $\eR(\Gc, N)_{t^{-1}} = N[t,t^{-1}]$. Then $\wt{f_{t^{-1}}} = f_*$.
\end{lemma}
\begin{proof}
Let $n \leq s(\Fc, \Gc)$. We note that $\Fc_n = M$ and $\Gc_n = N$. Let $m \in M$.
Then $$f_*(m)t^n = f(mt^n) = f_{t^{-1}}(mt^{n}) = f_*(m)t^{n}/1 = \wt{f}(m)t^n.$$
The result follows.
\end{proof}
We now give
\begin{proof}[Proof of Theorem \ref{psi}]
Let $\eR(\Fc, M)$ be  a MCM $\eR = \eR_A(I)$-module. Let $X$ be a graded complete resolution of $\eR(\Fc, M)$ with  $\coker(X^{-1} \rt X^0) = \eR(\Fc, M)$. So $\coker(X^{-1}_{t^{-1}} \rt X^0_{t^{-1}}) = M[t, t^{-1}]$. Taking degree zero component we get that $\Phi(X)$ is a complete resolution of $M$. Thus $\Psi(\eR(\Fc, M)) = M$.

Let $\eR(\Fc, M)$ and $\eR(\Gc, N)$  be   MCM $\eR$-modules and let $f \colon \eR(\Fc, M) \rt \eR(\Gc, N)$ be $\eR$-linear. Let $X$ (resp. $Y$) be complete resolutions of $\eR(\Fc, M)$(resp. $\eR(\Gc, N)$). Let $\wh{f} \colon X \rt Y$ be a lift of $f$. So we have a commutative diagram
\[
  \xymatrix
{
 X^{-1}
\ar@{->}[r]
\ar@{->}[d]^{u^{-1}}
 & X^{0}
\ar@{->}[d]^{u^0}
\ar@{->}[r]
& \eR(\Fc, M)
\ar@{->}[r]
\ar@{->}[d]^{f}
& 0
\\
 Y^{-1}
 \ar@{->}[r]
 & Y^0
 \ar@{->}[r]
& \eR(\Gc, N)
\ar@{->}[r]
&0
\
 }
\]
Localizing at $t^{-1}$ and by \ref{bella-3} and  \ref{loc-tilde} we obtain a commutative diagram
\[
  \xymatrix
{
 X^{-1}_{t^{-1}}
\ar@{->}[r]
\ar@{->}[d]^{u^{-1}_*[t,t^{-1}]}
 & X^{0}_{t^{-1}}
\ar@{->}[d]^{u^0_*[t,t^{-1}]}
\ar@{->}[r]
& M[t,t^{-1}]
\ar@{->}[r]
\ar@{->}[d]^{f_*[t,t^{-1}]}
& 0
\\
 Y^{-1}_{t^{-1}}
 \ar@{->}[r]
 & Y^0_{t^{-1}}
 \ar@{->}[r]
& N[t,t^{-1}]
\ar@{->}[r]
&0
\
 }
\]
Taking the degree zero component we obtain that $\Psi(f) = f_*$. The result follows.
\end{proof}
\begin{remark}
On might wonder whether we can directly define  \\ $\Psi \colon \CMS_\Z(\eR_A(I)) \rt \CMS(A)$ by the formula given in Theorem \ref{psi}. To do this  note that we have have to show there exists \textit{natural isomorphism}
$$\Psi(\Omega^{-1}_{\eR}(-)) \cong \Omega^{-1}_A(\Psi(-)).$$  Note that
$\Om^{-1}_{\eR}(\eR(\Fc, M)) = \eR(\Gc, \Omega^{-1}_A(M) \oplus H)$ where $H$ is a free $A$-module and $\Gc$ is an $I$-stable filtration on $\Omega^{-1}_A(M) \oplus H$. The \emph{naturality} of  this isomorphism $\Psi(\Omega^{-1}_{\eR}(-)) \cong \Omega^{-1}_A(\Psi(-))$ is not clear to this author.
\end{remark}
The following Corollary is immediate from Theorem \ref{psi}.
\begin{corollary}\label{cocaine}
(with hypotheses as in Theorem  \ref{psi}). We have
$$ \ker \Psi = \{ \eR(\Fc, M) \in \CMS_\Z(\eR_A(I)) \mid M \ \text{is free}\}. $$
\end{corollary}
We note the following
\begin{proposition}\label{cocoa}
(with hypotheses as in Theorem  \ref{psi}). The functor $\Psi$ is dense and full.
\end{proposition}
\begin{proof}
We first prove that $\Psi$ is dense. Let $M$ be any MCM $A$-module. Consider $\eR(I, M) = \bigoplus_{n \in \Z} I^nM$. Let $0 \rt Y \rt X \rt \eR(I, M) \rt 0$ be a MCM approximation of $\eR(I, M)$. As $X$ is a MCM $\eR = \eR_A(I)$-module we get that $X = \eR(\Fc, N)$ where
$N$ is a MCM $A$-module and $\Fc$ is an $I$-stable filtration on $N$, see \ref{bella-l}. We note that $t^{-1}$ is $Y$-regular. So $Y =  \eR(\Gc, U)$ where
$U$ is an $A$-module and $\Gc$ is an $I$-stable filtration on $U$, see \cite[3.1]{P-Itoh-g}. Looking at sufficiently negative degrees we obtain an exact sequence of $A$-modules $0 \rt U \rt N \rt M \rt 0$.
It follows that $U$ is an MCM  $A$-module.
We have $\projdim_{\eR} Y$ is finite. So we have an exact sequence
$0 \rt H_r \rt H_{r-1} \rt \cdots \rt H_0 \rt \eR(\Gc, U) \rt 0$ where $H_i$ are free $\eR$-modules. By looking at sufficiently negative components we get that $\projdim_A U < \infty$. As $U$ is also MCM $A$-module it follows that $U$ is free.
Thus $N = M \oplus U$. We have $\Psi(X) = N \cong M$ in $\CMS(A)$.

Next we show that $\Psi$ is full. Let $M, N$ be MCM $A$-modules and let $f \colon M \rt N$ be $A$-linear. Note $f$ induces an $\eR$-linear map $f^\sharp \colon \eR(I, M) \rt \eR(I, N)$. We take MCM approximations $0 \rt Y_M \rt X_M \rt \eR(I, M) \rt 0$ and $0 \rt Y_N \rt X_N \rt \eR(I, N) \rt 0$  of $\eR(I, M)$ and $\eR(I, N)$ respectively. We note that $f^\sharp$ induces a $\eR$-linear map $g \colon X_M \rt X_N$ such that we have a commutative diagram

\[
  \xymatrix
{
0
\ar@{->}[r]
 & Y_M
\ar@{->}[r]
\ar@{->}[d]^{u}
 & X_M
\ar@{->}[d]^{g}
\ar@{->}[r]
& \eR(I, M)
\ar@{->}[r]
\ar@{->}[d]^{f^\sharp}
& 0
\\
0
\ar@{->}[r]
 & Y_N
 \ar@{->}[r]
 & X_N
 \ar@{->}[r]
& \eR(I, N)
\ar@{->}[r]
&0
\
 }
\]
By our earlier argument there exists MCM $A$-modules $U, V$ and $I$-stable filtrations $\Fc, \Gc$ on $U$, $V$ respectively such that $X_M  = \eR(\Fc, U)$ and  $X_N = \eR(\Gc, V)$ respectively. Similarly there exists free $A$-modules $K, L$ and $I$-stable filtrations $\Fc^\prime, \Gc^\prime$ on $K$, $L$ respectively such that $Y_M  = \eR(\Fc, K)$ and  $Y_N = \eR(\Gc, L)$ respectively.
Taking components for sufficiently negative degree we obtain a commutative diagram
\[
  \xymatrix
{
0
\ar@{->}[r]
 & K
\ar@{->}[r]
\ar@{->}[d]^{u_*}
 & U
\ar@{->}[d]^{g_*}
\ar@{->}[r]
& M
\ar@{->}[r]
\ar@{->}[d]^{f}
& 0
\\
0
\ar@{->}[r]
 & L
 \ar@{->}[r]
 & V
 \ar@{->}[r]
& N
\ar@{->}[r]
&0
\
 }
\]
The two rows above split. We note that $U \cong M$ and $V \cong N$ in $\CMS(A)$ and under this correspondence $g_*$ is mapped to $f$.
 The result follows.
\end{proof}
We will need the following result.
\begin{lemma}\label{clinton}
Let $V = \bigoplus_{n \in \Z}V_n$ be a finitely generated graded $\eR = \eR_A(I)$-module with $V_n = 0$ for all $n \ll 0$.
Let $0 \rt Y \rt X \rt V \rt 0$ be a MCM-approximation of $V$. Then $X = \eR(\Fc, H)$ where $H$ is a free $A$-module and $\Fc$ is an $I$-stable filtration on $H$.
\end{lemma}
\begin{proof}
As $X$ is a MCM $\eR = \eR_A(I)$-module we get that $X = \eR(\Fc, N)$ where
$N$ is a MCM $A$-module and $\Fc$ is an $I$-stable filtration on $N$, see \ref{bella-l}. We note that $t^{-1}$ is $Y$-regular. So $Y =  \eR(\Gc, U)$ where
$U$ is an $A$-module and $\Gc$ is an $I$-stable filtration on $U$, see \cite[3.1]{P-Itoh-g}. As $V_n = 0$ for $n \ll 0$ we obtain that $U \cong N$. We have $\projdim_{\eR} Y$ is finite. So we have an exact sequence
$0 \rt H_r \rt H_{r-1} \rt \cdots \rt H_0 \rt \eR(\Gc, U) \rt 0$, where $H_i$ are free $\eR$-modules. By looking at sufficiently negative components we get that $\projdim_A U < \infty$. As $U$ is also MCM $A$-module it follows that $U$ is free. So $N$ is free. The result follows.
\end{proof}
We finally prove
\begin{proof}[Proof of Theorem \ref{fourth}]
As $\Psi$ is dense and full it follows that $\ov{\Psi}$ is dense and full. Thus it suffices to show $\ov{\Psi}$ is faithful. Set $\Dc =  \CMS_\Z(\eR_A(I))/\ker \Psi$. Suppose
$\beta \in \Hom_{\Dc}(\eR(\Fc, M), \eR(\Gc, N)$ be such that $\ov{\Psi}(\beta) = 0$.
Note  $\beta $ can be written as a left fraction
\[
\xymatrix{
\
&\eR(\Hc, E)
\ar@{->}[dl]_{u}
\ar@{->}[dr]^{f}
 \\
\eR(\Fc, M)
\ar@{->}[rr]_{\beta = fu^{-1}}
&\
&\eR(\Gc, N)
}
\]
with $\cone(u) \in \ker \Psi$. It follows that $\Psi(f) = 0$. Thus it suffices to show that
$f = 0$ in $\D$. We note that as $\Psi(f) = 0$ we get that $f_* \colon E \rt N$ factors through a free $A$-module.
After shifting we may assume that $\Hc_n = E$ for $n \leq 0$ and $\Hc_1 \neq E$.

Claim-1: We may assume that $\Gc_n = N$ for $n \leq 0$.

(Note we are not asserting $\Gc_1 \neq N$.)

Proof of Claim-1. Suppose $\Gc_n = N$ for $n \leq a$ with $a < 0$.
We note that $\Gc_n \subseteq \Gc_{n+a}$ for all $n \in \Z$ as $a < 0$. So we have an inclusion of MCM $\eR$-modules
$ \eR(\Gc, N) \xrightarrow{i} \eR(\Gc(a), N)$, say  with co-kernel $V$. We note that $V_n =0$ for $n \ll 0$. Taking  MCM approximations we have a triangle in  $\CMS_\Z(\eR_A(I))$
$$ \eR(\Gc, N) \xrightarrow{i}  \eR(\Gc(a), N) \rt X \rt \Omega^{-1}_{\eR}(\eR(\Gc, N)).$$
We note that by \ref{clinton} it follows that $X \in \ker \Psi$. So $i$ is an isomorphism in $\D$.
Thus it suffices to show that $g = i\circ f = 0$ in $\D$. Notice $g_* = f_* =0$ in $\CMS(A)$.

We have an exact sequence $0  \rt \eR(I, E) \xrightarrow{i_E} \eR(\Hc, E) \rt U \rt 0$. We note that $U_n = 0$ for $n \ll 0$.  Let $Y_E$ be MCM approximation of $\eR(I, E)$. Taking MCM approximations we get by \ref{clinton} that $Y_E \cong \eR(\Hc, E)$ in $\D$. Similarly if $Y_N$ is MCM approximation of $\eR(I, N)$ then $Y_N \cong \eR(\Gc, N)$.

We note that $f_* \colon E \rt N$ induces $\wh{f_*} \colon \eR(I, E) \rt \eR(I, N)$. We have an induced map $X(\wh{f}) \colon Y_E \rt Y_N$. As $f_*$ factors through a free $A$-module say $W$ we get that $\wh{f_*}$ factors through $\eR(I, W)$. So $X(\wh{f_*}) $ factors through $\eR(I, W)$. Thus $X(\wh{f_*}) = 0$ in $\CMS_\Z(\eR_A(I))$.

Note we have a commutative diagram
\[
  \xymatrix
{
 \eR(I,E)
\ar@{->}[r]^{i_E}
\ar@{->}[d]^{\wh{f_*}}
 & \eR(\Hc, E)
\ar@{->}[d]^{f}
\\
 \eR(I, N)
 \ar@{->}[r]^{i_N}
 & \eR(\Gc, N)
\
 }
\]
Taking MCM approximations we get a commutative diagram in $\CMS_\Z(\eR_A(I))$:
\[
  \xymatrix
{
 Y_E
\ar@{->}[r]^{X(i_E)}
\ar@{->}[d]^{X(\wh{f_*})}
 & \eR(\Hc, E)
\ar@{->}[d]^{f}
\\
 Y_N
 \ar@{->}[r]^{X(i_N)}
 & \eR(\Gc, N)
\
 }
\]
As discussed earlier $X(i_E)$ and $X(i_N)$ are isomorphisms in $\D$. Also $X(\wh{f_*}) = 0$ in $\CMS_\Z(\eR_A(I))$ and hence in $\D$. Thus $f = 0$ in $\D$. The result follows.
\end{proof}
\section{Proof of Theorem \ref{fifth}}
In this section we prove Theorem \ref{fifth}. We need a few preliminaries.
Let $(A,\m)$ be a Gorenstein local ring and let $I$ be an ideal in $A$.
Throughout  set  $\eR = \eR_I(A)$ and $G = G_I(A)$. We assume $\eR$(and so $G$) is Gorenstein. Set $\smod(\eR)$ (respectively) $\smod(G)$) to be the abelian category of consisting of graded $\eR$ (respectively $G$)-modules and degree zero graded homomorphisms. If $\mathcal{A}$ is a skeletally small exact (or triangulated) category  then let $G_0(\mathcal{A})$ denote its zeroth $K$-group.

\s\label{dev} Consider
$$\smod(t^{-\infty}, \eR) = \{ M \in \smod(\eR) \mid t^{-s}M = 0 \ \text{for some $s \geq 1$}\}. $$
We note that $\smod(t^{-\infty}, \eR)$ is an abelian category. It contains $\smod(G)$ as a full abelian subcategory. We also have that if $M \in \smod(t^{-\infty}, \eR)$ then $t^{-s}M = 0$ for some $s \geq 1$. It follows that $M$ has a finite filtration $M \supseteq t^{-1}M \supseteq t^{-2} M \supseteq \cdots \supseteq t^{-s + 1}M \supseteq t^{-s}M = 0$ with $t^{-i}M/t^{-i-1}M$ a $G$-module. So by devissage,  \cite[Chapter II, 6.3]{W}, we get that the natural map $G_0(\smod(G)) \rt G_0(\smod(t^{-\infty}, \eR))$ is an isomorphism.
Set $\CMa(G)$ to be the category of all graded  MCM $G$-modules. By \cite[13.2]{Y} (also see \cite[Chapter II, 7.6]{W}) the natural map $\CMa(G) \rt \smod(G)$ induces an isomorphism $G_0(\CMa(G)) \rt G_0(\smod(G))$.
\s \label{map} Consider $X \colon \smod(t^{-\infty},\eR) \rt \CMS_\Z(\eR) $ defined by the MCM approximation functor. Recall we have constructed a triangulated functor $\Psi \colon \CMS_\Z(\eR) \rt \CMS(A)$
with kernel consisting of those MCM $\eR$-modules $E$ such that $E = \eR(F,\Fc)$ for some free $A$-module $F$ and an $I$-stable filtration $\Fc$ on $F$. We show
\begin{proposition}\label{utah}
(with hypotheses as in \ref{dev}). If $M \in \smod(t^{-\infty}, \eR)$ then $X_M \in \ker \Psi$. Conversely if $E \in \ker \Psi$ then there exists $M \in \smod(t^{-\infty}, \eR)$ with $X_M = E$.
\end{proposition}
\begin{proof}
  Let $M \in \smod(t^{-\infty}, \eR)$. We have an exact sequence of $\eR$-modules $0 \rt Y \rt X_M \rt M\rt 0$ where $\projdim_{\eR} Y$ is finite. As $X_M$ is MCM $A$-module we get by \ref{bella-l} that there exists an MCM $A$-module $W$ and a $I$-stable filtration $\Fc$ on $W$  with $X_M = \eR( \Fc, W))$. As $t^{-1}$ is $X_M$-regular it is $Y$-regular. So $Y = \eR(V,\Gc)$ for some $A$-module $V$ and an $I$-stable filtration $\Gc$ on $V$. As $\projdim_{\eR}Y$ is finite it follows after taking a finite free resolution of $Y$ as an $\eR$-module  that $\projdim_A V$ is finite. As $M_n  = 0$ for all $n \ll 0$ it follows that $W = V$. So $\projdim_A W$ is finite. As $W$ is MCM we get $W$ is free. Thus $X_M \in \ker \Psi$.

 Conversely Let $E \in \ker \Psi$. As $E$ is an MCM $\eR$-module  there exists an  MCM $A$-module $W$ and a $I$-stable filtration $\Fc$ on $W$  with $E = \eR( \Fc, W))$, see \ref{bella-l}. As $E \in \ker \Psi$ we get $W$ is a free $A$-module. We may assume after shifting that $\Fc_n = W$ for $n <0$. Say $W = A^m$. Then we have an inclusion $i\colon \eR^m \rt E$. Set $ M = \coker i$. As $M_n = 0$ for $n \leq 0$ we get that $t^{-s}M = 0$ for some $s \geq 1$. Thus $M \in \smod(t^{-\infty}, \eR)$. We also have an exact sequence
 $$0 \rt \eR^m \xrightarrow{i} E \rt M \rt 0.$$
 As $E$ is MCM $\eR$-module and $\eR^m$ is free we get $X_M = E$ in $\CMS_\Z(\eR)$.
\end{proof}

\begin{remark}
Let $T = \bigoplus_{n \in \Z}T_n$ be a graded Noetherian ring. Let $\smod(T)$ denote the abelian category of finitely generated graded $T$-modules. Let $G_0(\smod(T))$ denotes its zeroth $K$-group.
Then there is an action of $\Z[u,u^{-1}]$ on $G_0(\smod(T))$  defined as $u [M] = [M(1)]$. If $T$ is Gorenstein then similarly there is a $\Z[u,u^{-1}]$ module structure on $G_0(\CMS_\Z(T))$ the  zeroth Grothendieck group of the stable category of all graded MCM $T$-modules. Note if $A$ is local and Gorenstein  then we define the $\Z[u,u^{-1}]$ module structure on $G_0(A)$ by setting $u[M] = [M]$.
\end{remark}
\begin{corollary}
\label{bhut}
(with hypotheses as in \ref{dev}). We have an exact sequence $$G_0(\CMS_\Z(G)) \xrightarrow{G_0(X)} G_0(\CMS_\Z(\eR)) \xrightarrow{\rho} G_0(\CMS(A)) \rt 0$$ of $\Z[u,u^{-1}]$-modules.
\end{corollary}
\begin{proof}
  We have a additive functor $X \colon \smod(t^{-\infty}, \eR) \rt \CMS_\Z(\eR)$ given by the MCM approximation functor whose image is $\ker \Psi$. We note that if $0 \rt U \rt V \rt W \rt 0$ is an exact sequence in $\smod(t^{-\infty}, \eR)$ then we have an exact sequence $0 \rt X_{U} \rt  X_V \rt X_W \rt 0$, see \cite[1.4]{KK}. This induces a triangle $X_{U} \rt X_V \rt X_W \rt \Omega^{-1}(X_U)$ in $\CMS_\Z(\eR)$.  So we get an additive map of abelian groups
  $G_0(\smod(t^{-\infty}, \eR)) \xrightarrow{G_0(X)} G_0(\CMS_\Z(\eR))$ whose image is $\mathcal{U} =\image (G(\ker \Psi) \rt G(\CMS_\Z(\eR))$ (by \ref{utah}).
  By \ref{dev} it follows that we have an additive map of abelian groups
  $G_0(\smod(G)) \xrightarrow{G_0(X)} G_0(\CMS_\Z(\eR))$ whose image is $\mathcal{U}$. Note $G_0(\smod(G)) = G_0(\CMa(G))$. Also note that $X_{G(m)} = \eR(m)$. Also $G_0(\CMS_\Z(G)) $ is the quotient of $G_0(\CMa(G))$ by the subgroup generated by $\{ [G(m)] \mid m \in \Z \}$. So we have an induced map $G_0(\CMS_\Z(G)) \xrightarrow{G(X)} G_0(\CMS_\Z(\eR))$ whose image is $\mathcal{U}$.

  As $\Psi$ is a triangulated functor we get an exact sequence of abelian groups $G_0(\ker \Psi) \rt G_0(\CMS_\Z(\eR)) \xrightarrow{\rho} G_0(\CMS(A)) \rt 0$. So we have an exact sequence of abelian groups
  $$G_0(\CMS_\Z(G)) \xrightarrow{G_0(X)} G_0(\CMS_\Z(\eR)) \xrightarrow{\rho} G_0(\CMS(A)) \rt 0.$$
  We show $G_0(X)$ and $\rho$ are  $\Z[u,u^{-1}]$-linear.
  We note that $$G_0(X)([uM]) = G_0(X)([M(1)]) = [X_{M(1)}] = [X_M(1)] = u[X_M] = uG_0(X)([M]). $$
   Similarly $G_0(X)([u^{-1}M]) = u^{-1}G_0(X)([M])$. So $G_0(X)$ is $\Z[u,u^{-1}]$-linear.
   Let $E  \in \CMS(\eR)$. We get by \ref{bella} that there exists an MCM $A$-module $W$ and a $I$-stable filtration $\Fc$ on $W$  with $E= \eR( \Fc, W))$. We note that $E(1) = \eR(\Fc(1), W)$.
   Note $\Psi(E(1)) = W$. We have $$\rho(u[E]) = \rho([E(1)]) = [W] = u[W] = u \rho([E]).$$
   Similarly   $\rho(u^{-1}[E]) = u^{-1} \rho([E]).$
   So $\rho$ is $\Z[u,u^{-1}]$-linear. The result follows.
\end{proof}

\s \label{rep} Let $T = \bigoplus_{n \in \Z}T_n$ be a graded Gorenstein ring and let $\CMS_\Z(T)$ be the stable category of graded MCM $A$-modules. We give the Grothendieck group $G_0(\CMS_\Z(T))$ a structure of $\Z[u,u^{-1}]$-module. Note if $f(u) \in \Z[u, u^{-1}]$ and $E \in \CMS_\Z(T)$ then $f(u)[E] = [W]$ where $W \in \CMS_\Z(T)$. We denote $[W]$ by $[f(u)E]$. We first note that $a\in \Z$ and $g(u) = au^i$ then $g(u)[E] = a[E(i)] $. Also note that $a[V] = [V^a]$ if $a > 0$ and $= [\Omega^{-1}(V)^{-a}]$ if $a < 0$ for any $V \in \CMS_\Z(T)$. We set $g(u)[E] = [g(u)(E)]$. From this it easily follows that
if $f(u) = \sum_{i = m}^{n}a_iu^i$ where $m,n\in \Z$ then $f(u)[E] = [\bigoplus_{i = m}^{n} (a_iu^i(E))]$.

 We then have
\begin{proposition}\label{rep-lemm}
(with hypothesis in \ref{rep}). An element $\xi \in G_0(\CMS_\Z(T))_{\Q(u)}$ is of the form $[W]/h(u)$ where $W \in \CMS_\Z(T)$ and $h(u) \in \Z[u,u^{-1}]$ is non-zero.
\end{proposition}
\begin{proof}
  We have $\xi = \sum_{i = 1}^{s}g_i(u)[E_i]/f_i(u)$ where $g_i(u), f_i(u) \in \Z[u,u^{-1}]$ and $f_i(u) \neq 0$ for all $i$. Then we have
  $$\xi = \frac{\sum_{i = 1}^{s}g_i^\prime(u)[E_i]}{f(u)} \text{where $f = f_1(u)\cdots f_s(u)$.} $$
  By our earlier discussion in \ref{rep} we have $\sum_{i = 1}^{s}g_i^\prime(u)[E_i] = [W]$ for some $W \in \CMS_\Z(T)$. The result follows.
\end{proof}

\s\label{rep-g}Consider the map $G_0(X) \colon G_0(\CMS_\Z(G)) \rt G_0(\CMS_\Z(\eR))$. let $M$ be an MCM $G$-module. We note that $$u^i[X_M] = [X_M(i)] = [X_{M(i)}] = [X_{u^iM}].$$
 Let $a \in \Z$. Then note that
 (i)  if $a > 0$ then $$a[X_M] = [X_M^a] = [X_{M^a}] = [X_{a M}].$$
 (ii) if $a< 0$ then $$a[X_M] = [\Omega_{\eR}^{-1}{(X_M^{-a})}] = [X_{\Omega_A(M)^{-a}}] = [X_{a M}].$$
It follows that if $f(u) \in \Z[u,u^{-1}]$ then $$f(u)[X_M] = [X_{f(u)M}].$$
We now give
\begin{proof}[Proof of Theorem \ref{fifth}]
By \ref{bhut} we have an exact sequence $$G_0(\CMS_\Z(G)) \xrightarrow{G_0(X)} G_0(\CMS_\Z(\eR)) \xrightarrow{\rho} G_0(\CMS(A)) \rt 0$$ of $\Z[u,u^{-1}]$-modules.
If $M \in \CMS(A)$ then note that $u[M] = [M]$. So $(u-1)[M] = 0$. Thus $G_0(A)_{\Q(u)} = 0$. So we have a surjection
$$ G_0(\CMS_\Z(G))_{Q(u)} \xrightarrow{G_0(X)_{\Q(u)}} G_0(\CMS_\Z(\eR))_{Q(u)}.$$
Let $\xi \in G_0(\CMS_\Z(G))_{\Q(u)}$ be such that $G_0(X)_{\Q(u)}(\xi) = 0$. By \ref{rep-lemm} we have $\xi = [M]/h(u)$ where $M \in \CMS_\Z(G)$ and $h(u) \in \Z[u,u^{-1}]$ is non-zero.
We have $0 = G_0(X)_{\Q(u)}(\xi)  =  [X_M]/h(u)$. So there exists $g(u) \in \Z[u,u^{-1}]$ non-zero such that $g(u)[X_M] = 0$. By \ref{rep-g} we get $[X_{g(u)M}] = 0$ in $G(\CMS_\Z(\eR))$. Choose $N$ with  $[N] = [g(u)M]$.
By \cite[2.4]{Th} there exists $U, V, W\in \CMS_\Z(\eR)$ and triangles
\begin{align*}
U &\rt V\oplus X_N \rt W \rt \Omega^{-1}(U), \quad \text{and} \\
U &\rt V \rt W \rt \Omega^{-1}(U).
\end{align*}
We consider the triangulated functor  $-\otimes \eR/t^{-1}\eR \colon \CMS_\Z(\eR) \rt \CMS_\Z(G)$. So we obtain triangles in $\CMS_\Z(G)$
\begin{align*}
U/t^{-1}U &\rt V/t^{-1}V\oplus X_N/t^{-1}X_N \rt W/t^{-1}W \rt \Omega^{-1}(U/t^{-1}U), \quad \text{and} \\
U/t^{-1}U &\rt V/t^{-1}V \rt W/t^{-1}W \rt \Omega^{-1}(U/t^{-1}U)
\end{align*}
It follows that $[X_N/t^{-1}X_N] = 0$ in $G(\CMS_\Z(G))$.
We have an exact sequence $0 \rt F \rt X_N \rt N \rt 0$ in $\eR$ with $F$ free $\eR$-module. After tensoring with $G$ we obtain
$$0 \rt N(+1) \rt F/t^{-1}F \rt X_N/t^{-1}X_N \rt N \rt 0.$$
In $G(\CMS_\Z(G))$  we obtain
$$(1-u)[N] = [X_N/t^{-1}X_N ] = 0. $$
So $[N] = 0$ in $G_0(\CMS_\Z(G))_{\Q(u)}$. Hence $[M] = 0$ and so $\xi = 0$ in \\ $G_0(\CMS_\Z(G))_{\Q(u)}$. Thus the map $G_0(X)_{\Q(u)}$ is bijective.
\end{proof}

\s Examples: We give several examples which illustrate our results of this section.
Let us recall our results. Let $(A,\m)$ be a  Gorenstein local ring and let $I$ be an ideal in $A$ with $\eR = \eR(I)$ Gorenstein (equivalently $G = G_I(A)$ is Gorenstein). We proved
there exists an exact sequence
 $$G_0(\CMS_\Z(G)) \xrightarrow{G_0(X)} G_0(\CMS_\Z(\eR)) \xrightarrow{\rho} G_0(\CMS(A)) \rt 0$$ of $\Z[u,u^{-1}]$-modules.
 This induces an isomorphism  of $\Q(u)$-vector spaces
 $$ G_0(\CMS_\Z(G))_{Q(u)} \xrightarrow{G_0(X)_{\Q(u)}} G_0(\CMS_\Z(\eR))_{Q(u)}.$$

(i) Assume the case when  $A$ is one-dimensional. Then

(a) $G_0(\CMS(\eR))$ is a finitely generated $\Z[u,u^{-1}]$-module.

(b) If $I$ is $\m$-primary then $\rank_{\Q(u)} G_0(\CMS_\Z(\eR))_{Q(u)} \leq \sharp(\Min (G))$  (here $\Min(G)$ is the set of minimal primes of $G$ and $\sharp(S)$ denotes the cardinality of a set $S$).

Proof of (a):  We note that $G_0(\smod(G))$ is generated as a $\Z[u,u^{-1}]$-module by $\{[G/P] \mid P \in \Spec^*(G) \}$ (here $\Spec^*(G)$ is the set of homogeneous primes of $G$). Also note that $G_0(\md(A))$  is generated as a $\Z$-module by $\{[A/P] \mid P \in \Spec(A) \}$.  Note $\Spec^*(G)$ and $\Spec(A)$ are finite sets. So (a) follows.

Proof of (b) Note the assertion follows from the following:

Claim: Let $U$ be a graded $G$-module of finite length. Then $[U]$ in $G_0(G)_{\Q(u)}$ is zero.

Proof of Claim: Taking a graded  filtration of $U$ we get that $[U] = h(u)[k]$ for some $h(u) \in \Z[u,u^{-1}]$.
Let $P$ be a minimal prime of $G$. Let $R = G/P$. Note  $R_0$ is  a field as $I$ is $\m$-primary. We  can choose $x\in R_s$  with $s > 0$ which is $R$-regular.
Either case we have an exact sequence  $0 \rt R(-\deg x) \rt R \rt R/(x) \rt 0$. So $(1- u^{-\deg x})[R] = [R/(x)] $. So $[k]$ is  zero in $G_0(\CMS_\Z(G))_{\Q(u)}$. The result follows.

(ii) Let $k$ be a field of characteristic not two and let $(V, Q)$ be a non-degenerate quadratic space over $k$. So there exists a basis $\{ x_1, \ldots, x_n \}$ of $V^*$ such that $Q = \sum_{i=1}^{n}a_ix_i^2$ and $a_i \neq 0$ for all $i$. Set $A = k[[X_1, \ldots, X_n]]/(Q)$. Then $A$ is of finite representation type, i.e., $A$ has only finitely many MCM modules up to isomorphism, see \cite[14.10]{Y}. Let $\m$ be the maximal ideal of $A$. Then $G = G_\m(A) = k[X_1, \ldots, X_n]/(Q)$. Furthermore the completion of $G$ is $A$. So $G$ is of finite representation type, i.e., $G$ has finitely many graded MCM modules up to isomorphism and shifts, see \cite[15.1]{Y}. It follows that $G_0(\CMS(\eR))$ is a finitely generated $\Z[u,u^{-1}]$-module.

Now assume that $\dim A$ is positive and even. Then $A$ has only one non-free MCM $A$-module. It follows that $G$ has only one non-free graded MCM-module, \cite[15.2]{Y}. Let $M$ be a nonfree MCM $A$-module. Then $G(M)$ is MCM and non-free $G$-module, see \cite[Theorem 16]{P-1}. We have an exact sequence $0 \rt M \rt F \rt M \rt 0$ of $A$-modules where $F$ is a free $A$-module.
This induces an exact sequence $0 \rt G(M)(-1) \rt G(F) \rt G(M) \rt 0$, see \cite[14.3, 14.4]{Y}. So $(1 + u^{-1})[G(M)] = 0$ in $G_0(\CMS_\Z(G))_Q$. It follows that $G_0(\CMS_\Z(G))_{\Q(u)} = 0$. So $G_0(\CMS_\Z(\eR))_{\Q(u)} = 0$.

(iii) Let $\bx = x_1,\ldots,x_d$ be a system of parameters of $A$. Set $I = (\bx)$. Then $G_0(\CMS_\Z(\eR))_\Q = G_0(\CMS(A))_\Q$ and $G_0(\CMS_\Z(\eR))_{\Q(u)} = 0.$

It suffices to prove $G_0(\CMS_\Z(G))_\Q = 0$. We note that $G = A/I[X_1, \ldots, X_d]$. We note that $G_{red} = k[X_1, \ldots, X_d] = R$. So $G_0(\smod(G)) = G_0(\smod(R))$. We note that the only graded MCM $R$-modules are $R(i)$ for some $i \in \Z$. Let $A/I = U_0 \supseteq U_1 \supseteq U_2 \supseteq \cdots \supseteq U_s  \supseteq 0$ be a filtration of $A/I$ as an $A$-module with quotients $k$. Tensoring with $G$ yields an equation $[G] = s[R]$ in $G_0(\smod(G))$. It follows that $G_0(\CMS(G))_\Q = 0$.
\section{Proof of Theorem \ref{equi}}
In this section we give a proof of Theorem \ref{equi}. Let $(A,\m)$ be a Gorenstein local ring and let $I$ be an  $\m$-primary ideal in $A$.
Throughout  set  $\eR = \eR_I(A)$ and $G = G_I(A)$. We assume $\eR$(and so $G$) is Gorenstein. We need a preparatory result first.
\begin{lemma}\label{rain}
Let $E \in \CMS^0_\Z(\eR)$.
Then there exists a free $A$-module $W$ and an $I$-stable filtration $\Fc$ on $W$ such that $\eR(\Fc, W) \cong  E$.
\end{lemma}
\begin{proof}
By \ref{bella-l}
  there exists an MCM $A$-module $W$ and an $I$-stable filtration $\Fc$ on $W$ such that $\eR(\Fc, W) \cong  E$. We have $E_{t^{-1}} = W[t,t^{-1}]$ an $A[t,t^{-1}]$-module. However $E \in \CMS^0_\Z(\eR)$. So $E_{t^{-1}} = W[t,t^{-1}]$ is a projective graded $A[t,t^{-1}]$-module. But $A[t,t^{-1}]$ is a $*$-local ring. So graded projectives are free. Thus $W$ is a free $A$-module. The result follows.
\end{proof}
We now give
\begin{proof}[Proof of Theorem \ref{equi}]
Let $\n$ and $\n^\prime$ be the maximal homogeneous  ideals of $\eR$ and $\wh{\eR}$-respectively. Then note that $\n \wh{\eR} = \n^\prime$. We have $k = \eR/\n = \wh{\eR}/\n^\prime = k \otimes_{\eR}\wh{\eR}$. It follows that if $Y$ is any graded $\eR$-module of finite length then $Y\otimes_{\eR}\wh{\eR} \cong Y$ as $\eR$-modules.
We note if $U, V \in \CMS_\Z^0(\eR)$ then $\sHom_{\eR}(U, V)$ has finite length as an $\eR$-module. So
$$\Hom_{\wh{\eR}}(U\otimes_{\eR}\wh{\eR}, V\otimes_{\eR}\wh{\eR}) \cong \Hom_{\eR}(U, V) \otimes_{\eR} \wh{\eR} \cong \Hom_{\eR}(U, V).$$
Thus $\Phi^\sharp $ is fully faithful.
We show that it is dense.
Let $E \in \CMS^0_\Z(\wh{R})$. By \ref{rain} there exists a free $\wh{A}$-module $W$ an $I\wh{A}$-stable filtration $\Fc$ on $W$ such that $\eR(\Fc, W) \cong  E$. By \ref{utah} there exists a $\wh{\eR}/(t^{-s})$-module $M$ with its MCM approximation $X_M = E$. But as $I$ is $\m$-primary we have $\eR/(t^{-s}) \cong \wh{\eR}/(t^{-s})$. So $M$ can be considered as an $\eR$-module. We note that if $Y$ is MCM approximation of $M$ as an $\eR$-module then $Y\otimes_\eR\wh{\eR}$ is the MCM approximation of $M$ as an $\wh{\eR}$-module. So $E \cong Y\otimes_\eR\wh{\eR}$ in $\CMS_\Z(\wh{\eR})$. Thus $\Phi^\sharp $ is dense. The result follows.
\end{proof}
\section{Proof of Theorem \ref{ri}}
Throughout this section $A$ and $B$ are Gorenstein local rings. In this section we give a proof of Theorem \ref{ri}. We need a few preliminary results.

\begin{lemma}\label{pet}
Assume $\dim A = d$.
The following assertions are equivalent:
\begin{enumerate}[\rm (i)]
  \item $A$ satisfies $R_i$.
  \item $\dim \sHom_A(M, M) \leq d -i - 1$ for all MCM $A$-modules.
\end{enumerate}
\end{lemma}
\begin{proof}
  (i) $\implies$ (ii): Let $M$ be a non-free MCM $A$-module. Let $P$ be a prime with $\height P = s$ with $P$ in the support of $\sHom_A(M, M)$. Then $\sHom_{A_P}(M_P, M_P) = \sHom_A(M, M)_P \neq 0$. So $M_P$ is not free as an $A_P$-module. Thus $A_P$ is not regular local. So $\height P > i$. It follows that $\dim \sHom_A(M, M) \leq d -i - 1$.

  (ii) $\implies$ (i): Let $P$ be a prime in $A$ with $\height P \leq  i$. Set $M = \Omega^d_A(A/P)$. As $\dim \sHom_A(M, M) \leq d -i - 1$ we get that $P \notin \Supp(\sHom_A(M, M))$.
  So $\sHom_A(M, M)_P = 0$. It follows that  $\sHom_{A_P}(M_P, M_P) = 0$. So $M_P$  is a free $A_P$-module. Therefore $\projdim_{A_P} \kappa(P)$ is finite. So $A_P$ is regular. Thus $A$ satisfies $R_i$.
\end{proof}
The following result is easy to prove.
\begin{proposition}
\label{baby} Let $\phi \colon  R \rt S$ be a ring homomorphism with $S$ a finite $R$-module via $\phi$. Let $M$ be a finite $S$-module. Consider $M$ as an $R$-module via $\phi$. Then $M$ is a finite $R$-module and $\dim_R M = \dim_S M$ (here $\dim_R M$ and $\dim_S M$ are the dimension of $M$ as an $R$-module (and respectively as an $S$-module)).
\end{proposition}
We now give
\begin{proof}[Proof of Theorem \ref{ci}]
As $A$ (and so $B$) are not hypersurfaces we get by \ref{first} that $\dim A = \dim B = d$(say). First assume that $A$ satisfies $R_i$. Let $M$ be an MCM $B$-module. Let $N = \Phi^{-1}(M)$. Then
we have an isomorphism $\sHom_A(N, N) \cong \sHom_B(M, M)$. Let $Z_M = $ center of $\sHom_B(M,M)$ and let \\ $Z_N = $ center of $\sHom_A(N,N)$. Then $Z_M \cong Z_N$. Furthermore $Z_M$ (respectively $Z_N$) are finite $B$ (respectively $A$)-module). By \ref{pet} we get $\dim_A \sHom_A(N, N) \leq d - i -1$. By \ref{baby} we get $\dim_{Z_N} \sHom_A(N, N) \leq d - i -1$. As $Z_M \cong Z_N$ and
$\sHom_A(N, N) \cong \sHom_B(M, M)$ we get that $\dim_{Z_M} \sHom_B(M,M) \leq d - i -1$. So again by \ref{baby} we get $\dim_B \sHom_B(M, M) \leq d -i - 1$. The result follows from \ref{pet}.
\end{proof}

\section{Proof of Theorem \ref{ci}}
In this section we give
\begin{proof}[Proof of Theorem \ref{ci}]
Assume $A$ is a complete intersection on the punctured spectrum of $A$. Let $M$ be a MCM $A$-module. Let $Z_M = $ center of $\sHom_A(M, M)$. We note that $Z_M$ is a finite $A$-module.
Let $\q$ be a non-maximal prime ideal of $Z_M$. We note that for all $i \geq 1$, $\Ext^i_A(M, M)$ is a $Z_M$-module.

Claim-1: If $\ell_{(Z_M)_\q}(\Ext^i_A(M,M)_\q)$ is finite for all $i \geq 1$ then there exists $c$ such that $\limsup_n \ell_{(Z_M)_\q}(\Ext^i_A(M,M)_\q) /n^c$ is finite.

Let $P = \q \cap A$. Then $P \neq \m$ as $Z_M$ is a finite $A$-module. As $A_P$ is a complete intersection it follows that $\Ext^{\geq 1}_{A_P}(M_P, M_P)$ is a finite $A_P[t_1,\ldots, t_m]$-module for some variables $t_1,\ldots,t_m$ of degree $2$, see \cite[4.9]{AGP}. As $\Ext^i_A(M,M)_\q$ is a further localization of $\Ext^i_{A_P}(M_P, M_P)$ it follows that $\Ext^{\geq 1}_{(Z_M)_\q}(M, M)$ is a finite $(Z_M)\q[t_1,\ldots, t_m]$-module. The result follows.

Now let $Q$ be non-maximal prime of $B$. Set $N = \Omega^{\dim B}_B(B/Q)$. Note for $i \geq 1$ the $B_Q$-module $\Ext^i_B(N, N)_Q$ has finite length. Let $Z_N = $ center of $\sHom_B(N,N)$. Let $\q_1, \ldots \q_r$ be all the primes of $Z_N$ with $\q_i \cap B = Q$. We note that $\q_j$ are not maximal ideals of $Z_N$. Let $M =\Phi^{-1}(N)$. Then $Z_M \cong Z_N$. Also note that for $i \geq 1$ we have
$$\Ext^i_A(M, M) \cong \sHom_A(\Omega_A^i(M), M) \cong \sHom_B(\Omega_A^i(N), N) \cong \Ext^i_B(N, N). $$
It follows from Claim-1 that for all $j = 1, \ldots,r$ there exists  $c_i$ such that  $$\limsup_n \ell_{(Z_N)_{\q_i}}(\Ext^i_B(N,N)_{\q_j}) /n^{c_j} \quad \text{ is finite}.$$
Let $c = \max\{c_j \mid j = 1, \ldots, r \}$. Then by \ref{bella} it follows that  $$\limsup_n \ell_{B_Q}(\Ext^i_B(N,N)_{Q}) /n^{c} \quad \text{ is finite}.$$
So $\cx_{B_Q}(\kappa(Q))$ is finite. It follows that $B_Q$ is a complete intersection, see \cite{G-ci} (also see \cite[8.1.2]{A}). Thus $B$ is a complete intersection on the punctured spectrum of $B$.
\end{proof}
\section{Proof of Theorems \ref{rigid} and \ref{rigid-ci}}
In this section we give proofs of Theorems \ref{rigid} and \ref{rigid-ci}. We first make a definition.
\begin{definition}\label{tri-f}
 Let $\C$ be a triangulated category with shift operator $\Sigma$. Assume $\C$ is skeletally small. Let $I(\C)$ denote the set consisting of isomorphism classes of objects in $\C$.

We say a function $\xi \colon I(\C) \rt \Z$ is a \emph{triangle
 function} on $\C$, see \cite[Introduction]{P-liason} if it satisfies the following properties:
\begin{enumerate}
\item
$\xi([M]) \geq 0$ for all $M \in \C$.
\item
$\xi([M]) = 0$ if and only if $M = 0$ in $\C$ .
\item
$\xi([M_1 \oplus M_2]) = \xi([M_1]) + \xi([M_2])$ for all $M_1, M_2 \in \C$.
\item
\emph{
(sub-additivity)} If $M \rt N \rt L \rt \Sigma (M)$ is an exact triangle in $\C$ then
\begin{enumerate}
\item
$\xi([N]) \leq \xi([M]) + \xi([L])$.
\item
$\xi([L]) \leq \xi([N]) + \xi([\Sigma(M)])$.
\item
$\xi([\Sigma(M)]) \leq \xi([L]) + \xi([\Sigma(N)])$.
\end{enumerate}
\end{enumerate}
\end{definition}
\begin{remark}\label{rotation}
\begin{enumerate}[\rm (i)]
\item
Since rotations of exact triangles are exact it follows that if $\xi$ satisfies (4)(b) for all exact triangles then it will also satisfy 4(a),(c).
\item
Axiom (3) implies that $\xi([M]) = 0$ if $M = 0$. However note that axiom (2) also implies that if $\xi([M]) = 0$ then $M = 0$.
\end{enumerate}
\end{remark}

\s If $\C = \CMS(A)$ for a Gorenstein ring then it is not difficult to check that $\xi_j([M]) = \ell_A(\Tor^A_j(M, k))$ are triangle functions on $\CMS(A)$ for all $j \geq 1$.

Next we show
\begin{theorem}
\label{infy}
Suppose $\C$ is a skeletally small triangulated category and let $\D$ be a proper thick subcategory such that the quotient $\C/\D$ has a triangle function $\xi$. Then $(\C, \D)$ is a rigid pair.
\end{theorem}
\begin{proof}
Let $t_s  \colon N_s \rt E_s \rt M \rt N_s[1]$
 with $ t_s$   a triangle and $ N_s \in \D$. Then  $M \cong E_s$ in $\C/\D$. Note $\alpha_{\C, \D}(E) \leq \xi([M]) $. It follows that $r_0(\C,\D, M)$ is finite for all $M \in \C$. Similarly $l_0(\C,\D, M)$ is finite for all $M \in \C$. The result follows.
\end{proof}

Next we give
\begin{proof}[Proof of Theorem \ref{rigid}]
We note that we have an induced triangle functor $$\ov{\Psi} \colon \CMS(A)/\ker \Psi \rt \CMS(B),$$ whose kernel is zero. Let $\xi$ be any triangle function on $\CMS(B)$. Note $\CMS/\ker \Psi$ is skeletally small. Define $\eta \colon I(\CMS/\ker \Psi) \rt \Z$ with
$\eta([M]) = \xi(\Psi(M))$.

Claim $\eta$ is a triangle function on $\CMS/\ker \Psi$. We verify the axioms (1)-(4) of
definition \ref{tri-f}.

Clearly $\eta$ satisfies (1) and (3).

(2) If $\eta[M]) = 0$ then $\xi(\Psi(M)) = 0$. As $\xi$ is a triangle function on $\CMS(B)$ we get $\Psi(M) = 0$. So $M \in \ker \Psi$. Thus $M = 0$ in $\CMS(A)/\ker \Psi $.

(3) Let $M \rt N \rt L \rt \Omega^{-1}(M)$ be a triangle in $\CMS(A)/\ker \Psi $. As $\ov{\Psi}$ is a triangle function we obtain that
$$\Psi(M) \rt \Psi(N) \rt \Psi(L) \rt \Omega^{-1}(\Psi(M))$$
is a triangle in $\D$. So we obtain
$$\xi(\Psi(L)) \leq \xi(\Psi(N) + \xi(\Omega^{-1}(\Psi(M)).$$
So we obtain
$\eta([L]) \leq \eta([N]) + \eta([\Om^{-1}(M)])$.
Thus $\eta$ satisfies 4(b).
Since rotations of exact triangles are exact it follows that if $\xi$ satisfies 4(a),(c).
Thus $\eta$ is a triangle function on $\CMS(A)/\ker \Psi$.
By \ref{infy} we get that $(\CMS(A),\ker \Psi)$  is a rigid pair.
\end{proof}

\s \emph{Examples:}

(1) Let $(A, \m)$ be a Henselian Gorenstein local ring with a prime $P \neq \m$ such that $A_P$ is  not regular.
Consider the map $\Psi \colon \CMS(A) \rt \CMS(A_P)$ given by $M \rt M_P$. Clearly $\Psi$ is a triangulated functor. If $M$ is free on the punctured spectrum of $A$ then $\Psi(M) = 0$. So $\ker \Psi \neq 0$. Set $N = \Omega^d_A(A/P)$ where $d  = \dim A$. Then note that
$N_P = \Omega^d_{A_P}(\kappa(P))$ is not a free $A_P$ module as $A_P$ is not regular. So $\ker \Psi \neq \CMS(A)$.
Thus $\ker(\Psi)$ is a proper thick subcategory of $\CMS(A)$. By Theorem \ref{rigid} we get that $(\CMS(A), \ker\Psi)$ is a rigid pair.

(2) Let $(A, \m)$ be a Henselian Gorenstein local ring which is not a complete intersection. Let
$\bx = x_1, \ldots, x_c \in \n^2$ be an $A$-regular sequence. Set $B = A/(\bx)$. Let $\Psi \colon \CMS(B) \rt \CMS(A)$ be the MCM approximation functor.
We have $\Psi$ is a triangulated functor, see \cite[p.\ 740]{BJM}. By \cite[3.1]{AGP},  there exists non-free  MCM $B$-module $M$  such that $\projdim_A M < \infty$. So the MCM approximation of $M$ as an $A$-module is free. Thus $\Psi(M) = 0$. Thus $\ker \Psi \neq 0$. Now let $N = \Omega^{\dim B}_B(k)$. If $\projdim_A  N < \infty$ then it follows from Gulliksen \cite[3.1]{G} that $\cx_B k \leq c$.   So $B$ is a complete intersection, see \cite{G-ci} (also see \cite[8.1.2]{A}). Thus $A$ is a complete intersection which is a contradiction. It follows that $\projdim_A  N = \infty$. So the MCM  approximation of $N$ is not free. Ihus $\Psi(N) \neq 0$. Thus $\ker \Psi \neq \CMS(B)$.
Therefore $\ker(\Psi)$ is a proper thick subcategory of $\CMS(B)$. By Theorem \ref{rigid} we get that $(\CMS(B), \ker\Psi)$ is a rigid pair.

\s \label{t-f} Let $\C$ be an essentially small triangulated category  with shift operator $\Sigma$ and let $I(\C)$ be the set of isomorphism classes of objects in $\C$. By a \emph{weak triangle function}, see \cite[2.1]{P-reg} on $\C$ we mean a function $\xi \colon I(\C) \rt \Z$ such that
\begin{enumerate}
  \item $\xi(X) \geq 0$ for all $X \in \C$.
  \item $\xi(0) = 0$.
  \item $\xi(X \oplus Y) = \xi(X) + \xi(Y)$ for all $X, Y \in \C$.
  \item $\xi(\Sigma X ) = \xi(X)$ for all $X \in \C$.
  \item If $X \rt Y \rt Z \rt \Sigma X $ is a triangle in $\C$ then
   $\xi(Z) \leq \xi(X) + \xi(Y)$.
\end{enumerate}
\s Set $$\ker \xi = \{ X \mid \xi(X) = 0 \}.$$
The following result cf., \cite[2.3]{P-reg} is essentially only an observation.
\begin{lemma}
\label{ker-lemma}(with hypotheses as above)
$\ker \xi $ is a thick subcategory of $\C$.
\end{lemma}
We now prove.
\begin{proposition}\label{pardesi}
Suppose $\xi$ is a weak triangle function on $\C$. Then $\xi$ induces a triangle function $\ov{\xi}$ on $\C/\ker \xi$ as $\ov{\xi}([M]) = \xi(M).$
\end{proposition}
\begin{proof}
  Set $\D = \C/\ker \xi$. We assume that $\D \neq 0$ otherwise there is nothing to prove.

   We  first have to prove that if $\beta \colon X \rt Y$  is an isomorphism in $\D$ then $\xi(X) = \xi(Y)$.
  Note  $\beta $ can be written as a left fraction
\[
\xymatrix{
\
&Z
\ar@{->}[dl]_{u}
\ar@{->}[dr]^{f}
 \\
X
\ar@{->}[rr]_{\beta = fu^{-1}}
&\
&Y
}
\]
As $u, \beta$ are isomorphisms in $\D$ we get $f$ is also an isomorphism in $\D$. Thus  $\cone(u), \cone(f) \in \ker\xi$.
We have a triangle
$$Z \rt X  \rt \cone(u) \rt\Sigma(Z).$$
Rotating we have triangles
$$  X \rt \cone(u) \rt \Sigma(Z) \rt \Sigma(X) \quad \text{and} \cone(u) \rt \Sigma(Z) \rt \Sigma(X) \rt \Sigma(\cone(u)).$$
By the first triangle we get $\xi(\Sigma(Z)) \leq \xi(\Sigma(X)$. So $\xi(Z) \leq \xi(X)$. Similarly by the second triangle we obtain $\xi(X) \leq \xi(Z)$. So $\xi(Z) = \xi(X)$.
Similarly by considering $f$ we get that $\xi(Z) = \xi(Y)$. So $\xi(X) = \xi(Y)$.

We now prove that $\ov{\xi}$ satisfies the axioms of a triangle function, see \ref{tri-f}. The axioms (1) and (3) are trivial to verify. If $\ov{\xi}(M) = \xi(M) = 0$ then $M \in \ker \xi$ and so is zero in $\D$.
Let  $t \colon M \rt N \rt L \rt \Sigma (M)$ is an exact triangle in $\D$. Then it is isomorphic in $\D$ to the image in $\D$ of an exact triangle in $\C$. Thus 4(b) holds by property (4) and (5) of \ref{t-f}.  As rotations of exact triangles are exact we get that $\ov{\xi}$ satisfies 4(a), (c). The result follows.
\end{proof}
We have
\begin{corollary}
\label{dont} Let $\C$ be a triangulated category and let $\xi \colon \C \rt \Z$ be a weak triangle function. Assume that $\ker \xi$ is a proper thick subcategory of $\C$. Then $(\C, \ker \psi)$ is a rigid pair.
\end{corollary}
\begin{proof}
  This follows from Proposition \ref{pardesi} and \ref{infy}.
\end{proof}

We now give
\begin{proof}[Proof of Theorem \ref{rigid-ci}]
Let $M \in  \CMS^{\leq i}(A)$. Let $\beta_n(M) = \ell(\Tor^A_n(M, k)$. Then it is known that $e_M(n) =  \beta_{2n}(M)$ and $o_M(n) =  \beta_{2n +1}(M)$ are both polynomial functions with same normalized leading term $\eta(M)$. Furthermore the growth of $e_M$ and $o_M$ is of order $n^{\cx(M)-1}$, see \cite[9.2.1]{A}.
Set $\xi(M) = \lim_{n \rt \infty} i!e_M(n)/n^{i-1} = \lim_{n \rt \infty} i!o_M(n)/n^{i-1}$. Then note that $\xi(M) = 0$ if and only if $M \in \CMS^{\leq i-1}(A)$.

We show that $\xi \colon \CMS^{\leq i}(A) \rt \Z$ is a weak triangle function. We verify the conditions in \ref{t-f}.
The axioms (1)-(4) are trivial to verify.
Let $M \rt N \rt L \rt \Omega^{-1}(M) \rt 0$ be a triangle in $\CMS(A)$. Then by construction of triangles in $\CMS(A)$ we have an exact sequence $0 \rt N \rt L \oplus F \rt \Omega^{-1}M \rt 0$ where $F$ is a free $A$-module. So we obtain $e_L(n) \leq e_N(n) + o_M(n)$ for all $n \geq 1$. The result follows.

As noted earlier $\ker \xi = \CMS^{\leq i-1}(A)$. So by \ref{dont} we get that \\ $(\CMS^{\leq i}(A), \CMS^{\leq i-1}(A))$ is a rigid pair.
\end{proof}

\section{Proof of Theorems \ref{tate} and \ref{tate-trivial}}
In this section we give proofs of Theorems \ref{tate} and \ref{tate-trivial}.
We first give
\begin{proof}[Proof of Theorem \ref{tate}]
We may assume that $M$ has no free summands. Let $\Fb$ be a minimal complete resolution of $M$. We index complexes homologically.
Let $\Fb \colon \cdots \rt \Fb_i \xrightarrow{\partial_i} \Fb_{i-1} \rt \cdots.$ We have $\image \partial_i = \Omega^i(M)$.

(1) \
(i) $\implies$ ii:

 Suppose $M \in \ker \alpha_X$. Fix $n \in \Z$. We have a surjection $\eta_n \colon \Fb_{n}\otimes X \rt \Omega^n(M) \otimes X \rt 0$ and a map $\epsilon_n \colon \Omega^n(M) \otimes X \rt \Fb_{n-1}\otimes X $.

Claim-1  The sequence
$$\mathcal{C} \colon  0 \rt  \Omega^n(M) \otimes X  \xrightarrow{\epsilon_n} \Fb_{n-1}\otimes X \rt \Fb_{n-2}\otimes X   \cdots  \rt \Fb_{i}\otimes X \rt \cdots.$$
is exact.

It suffices (and necessary) to prove $\epsilon_n$ is injective and $\mathcal{C}$ is exact at $\Fb_{n-1}\otimes X$.

Let $u \in \ker \epsilon_n$. We note that there exists $v \in \Fb_n\otimes X$  with $\eta_n(v) = u$. We have $\partial_n\otimes X(v) = \epsilon_n(\eta_n(v)) = 0$. As $\Fb\otimes X$ is exact there exists $w \in \Fb_{n+1}\otimes X $ with $\partial_{n+1}\otimes X(w) = v$. It follows that $u = 0$.  So $\epsilon_n$ is injective.

Next we show $\mathcal{C}$ is exact at $\Fb_{n-1}\otimes X$.
It is clear that $\partial_{n-1}\otimes X\circ \epsilon_n = 0$. So $\image \epsilon_n \subseteq \ker \partial_{n-1}\otimes X$. Let $u \in \ker \partial_{n-1}\otimes X$. As $\Fb\otimes X$ is exact there exists $w \in \Fb_{n}\otimes X $ with $\partial_{n}\otimes X(w) = u$. It follows that $u = \epsilon_n(\eta_n(w))$.  So $\ker \partial_{n-1}\otimes X  \subseteq \image \epsilon_n$. The result follows.

As $\mathcal{C}$ is exact and $X$ is MCM it follows that $ \Omega^n(M) \otimes X$ is MCM.

The assertion (ii)$ \implies$ (iii) is trivial.

(iii) $\implies$ (i). We have an exact sequence $ 0 \rt \Omega^{n+1}(M) \rt \Fb_{n} \rt \Omega^n(M) \rt 0$. Therefore we obtain
an exact sequence
$$ 0 \rt \Tor^A_1(\Omega^n(M), X) \rt \Omega^{n+1}(M)\otimes X \rt \Fb_{n}\otimes X \rt \Omega^n(M)\otimes X \rt 0. $$
As $X$ is free on the punctured spectrum we get $\Tor^A_1(\Omega^n(M), X)$ has finite length. Also $\depth \Omega^{n+1}(M)\otimes X > 0$. It follows that $\Tor^A_1(\Omega^n(M), X) = 0$. Thus for all $n \in \Z$ the sequence $ 0  \rt \Omega^{n+1}(M)\otimes X \rt \Fb_{n}\otimes X \rt \Omega^n(M)\otimes X \rt 0$ is exact. So $\Fb\otimes X$ is exact. Thus $M \in \ker \alpha_X$

(2) \
(i) $\implies$ ii: Let $M \in \ker \beta_X$.

Fix $n \in  \Z$. The exact sequence $\mathcal{C} \colon \rt \Fb_i \rt \cdots \rt \Fb_n \rt \Omega^n(M) \rt 0 $ induces a sequence
$$\Hom(\mathcal{C}, X) \colon 0\rt \Hom(\Omega^n(M), X) \rt \Hom(\Fb_n, X) \rt \cdots \rt \Hom(\Fb_i, X) \rt \cdots.$$
As $\Hom(\Fb, X)$ is exact and $\Hom(-, X)$ is left exact it follows that $\Hom(\mathcal{C}, X) $ is exact.
As $\Hom(\Fb_i, X) $ is MCM $A$-module for all $i$ it follows that $\Hom(\Omega_n(M), X)$ is MCM.

The assertion (ii)$ \implies$ (iii) is trivial.

(iii) $\implies$ (i). We have an exact sequence $ 0 \rt \Omega^{n+1}(M) \rt \Fb_{n} \rt \Omega^n(M) \rt 0$. Therefore we obtain
an exact sequence
\begin{align*}
   0 &\rt \Hom(\Omega^n(M), X) \rt \Hom(\Fb_n, X) \rt \Hom(\Omega^{n+1}(M), X)  \\
   &\rt \Ext_A^1(\Omega^n(M), X) \rt 0.
 \end{align*}
We assert that $U = \Ext_A^1(\Omega^n(M), X) = 0$. If not then note $U$ has finite length.
Counting depths as $\depth \Hom(\Omega^i(M), X) \geq 3$ for all $i$ and $\Hom(\Fb_n, X)$ is MCM, it follows that $\depth U \geq 1$ which is a contradiction.

(3) This is similar to (2).
\end{proof}

Next we give
\begin{proof}[Proof of Theorem \ref{tate-trivial}]
(1)  Suppose if possible  there exists $n_0$ such that \\ $\depth \Omega^n(M)\otimes X  \geq 1$ for all $n \geq n_0$. Then by an argument similar to proof of Theorem \ref{tate}(1)
 (iii) $\implies$ (i), we get $\Tor_1^A(\Omega_n(M), X) = 0$ for all $n \geq n_{0}$. So $\Tor^A_i(M, X) = 0$ for all $i \gg 0$. As $A$ has trivial Tor-vanishing we get $M$ or $X$ has finite projective dimension which is a contradiction.

(2)(a) Suppose if possible  there exists $n_0$ such that $\depth \Hom_A(\Omega^n(M), X)  \geq 3$ for all $n \geq n_0$. Then by an argument similar to proof of Theorem \ref{tate}(2) (iii) $\implies$ (i), we get $\Ext^1_A(\Omega^n(M), X) = 0$ for all $n \geq n_{0}$. So $\Ext_A^i(M, X) = 0$ for all $i \gg 0$. As $A$ has trivial Ext-vanishing we get $M$  has finite projective dimension or $X$ has finite injective  dimension. We have $M$ is non-free MCM. So $\projdim M = \infty$. Thus $X$ has finite injective dimension. But $A$ is Gorenstein. So we get $\projdim X < \infty$, a contradiction.

\end{proof}

\end{document}